\documentclass[12pt, a4paper]{article}
\usepackage{graphicx} 
\usepackage{amssymb}
\usepackage{mathrsfs}
\usepackage{amsmath}
\usepackage{amsthm}
\usepackage{appendix}
\usepackage{geometry}
\usepackage{circledsteps}
\usepackage{indentfirst}
\usepackage{xcolor}
\newtheorem{statement}{Statement}[section]

\theoremstyle{definition}
\newtheorem{definition}[statement]{Definition}

\theoremstyle{plain}
\newtheorem{theorem}[statement]{Theorem}
\newtheorem{lemma}[statement]{Lemma}
\newtheorem{proposition}[statement]{Proposition}

\newtheorem{remark}[statement]{Remark}

\numberwithin{equation}{section}

\allowdisplaybreaks[4]
\begin{document}

\title{Ergodicity of stochastic reaction-diffusion equations on unbounded domains  driven by space-time white noise}

\author{Shijie Shang$^{1}$, Jianliang Zhai$^{1}$,
		Tusheng Zhang$^{1,2}$}
	\footnotetext[1]{\, School of Mathematics, University of Science and Technology of China, Hefei, China. Email: sjshang@ustc.edu.cn (Shijie Shang), zhaijl@ustc.edu.cn (Jianliang Zhai).}
	\footnotetext[2]{\, Department of Mathematics, University of Manchester, Manchester M13 9PL, United Kingdom. Email: tusheng.zhang@manchester.ac.uk }

\date{\today}

\maketitle

\begin{abstract}
\noindent
We consider the stochastic reaction-diffusion equation on the whole space:
\begin{align*}
	\left\{
	\begin{aligned}
		du(t,x) &=\frac{1}{2}\partial_{xx} u(t,x) dt+b(u(t,x))dt+ \sigma(u(t,x)) W(dt,dx),\quad t\geq 0,\ x\in \mathbb{R},\\
		u(0,x)&=u_0(x), \quad x\in \mathbb{R},
	\end{aligned}
	\right.
\end{align*}
where $W(dt,dx)$ is a space-time white noise, $b$, $\sigma$ are measurable coefficients.
We first show that the solution is not strong Feller, and then establish the existence and uniqueness of invariant measures, 
exponential mixing as well as irreducibility for the solutions. 
To overcome the difficulties caused by the unbounded domain, we design special controls and controlled equations to prove the irreducibility. To obtain the exponential mixing property under the dissipative condition
$$(b(x)-b(y))(x-y)\leq -\alpha (x-y)^2,$$
the obstacle is the lack of the It\^{o} formula/energy equality. To circumvent the problem, we manage to find a new way to fully exploit comparison principles, which we believe could be useful for other type of  stochastic partial differential equations driven by multiplicative space-time noise. We note that the dissipative condition allows the coefficients to be of polynomial, even exponential growth. There exist plenty of models that satisfy the dissipative condition, including the Allen-Cahn type equations. 

To the best of our knowledge, this is the first paper to establish the ergodicity, exponential mixing and irreducibility  of stochastic reaction-diffusion equations (SRDEs) driven by multiplicative space-time noise on unbounded domains. The results on exponential mixing are also new for (SRDEs) driven by multiplicative space-time noise on bounded domains.
\par\vspace{3mm}

\noindent\textbf{Keywords:} Stochastic reaction-diffusion equations;  space-time white noise; strong Feller property; invariant measures; exponential mixing; irreducibility; unbounded domains.
\vskip 0.3cm
\noindent
	{\bf AMS Subject Classification:} Primary 60H15;  Secondary 35R60.
\end{abstract}
\tableofcontents
\section{Introduction}
In this paper, we consider the stochastic reaction-diffusion equation on the whole space:
\begin{align}\label{3.1}
	\left\{
	\begin{aligned}
		d u(t,x) &=\frac{1}{2}\partial_{xx} u(t,x)dt +b(u(t,x))dt + \sigma(u(t,x)) W(dt,dx),\quad t\geq 0, \ x\in \mathbb{R},\\
		u(0,x)&=u_0(x), \quad x\in \mathbb{R},
	\end{aligned}
	\right.
\end{align}
where the initial value $u_0$ is a deterministic function and $ W(dt,dx)$ is a space-time white noise on a probability space $(\Omega, {\cal F}, \{{\cal F}_t\}_{t\geq 0}, \mathbb{P})$. Here $\{{\cal F}_t\}_{t\geq 0}$
is the filtration satisfying the usual conditions generated by the white noise $W$. The coefficients
$b(\cdot), \sigma(\cdot): \mathbb{R}\rightarrow \mathbb{R}$ are deterministic
measurable functions.  

There is a great amount of literature devoted to the study of stochastic reaction-diffusion equations (SRDEs). Numerous works address well-posedness, asymptotic behaviour, ergodicity, and related topics; see \cite{CE-1,COTX26,PG-2,PG-3,DKZ,KKMS23,M93} and references therein for the case of SRDEs on bounded domains. 
For the case of SRDEs on unbounded domains, related problems have also been studied in \cite{CE24,CFHS25,CKNP22,FKN24,FKN26,GL20,JO26,S25,ST24,SZ2}.

In this paper,  we establish the ergodicity, exponential mixing as well as irreducibility of stochastic reaction-diffusion equations on unbounded domains. Ergodicity concerns the long-time behaviour of stochastic systems and is one of the most important research topics in the study of stochastic partial differential equations(SPDEs). Irreducibility plays very important roles in the study of large deviations, recurrence properties, strong maximal principles of the generators  and the ergodicity of SPDEs.
Ergodicity and irreducibility for stochastic evolution equations and stochastic partial differential equations (SPDEs) have been studied extensively. For instance, with the irreducibility in hand, one can obtain the ergodicity by proving the strong Feller property (see \cite{PG-3,PZ,Z}), or the asymptotic strong Feller (see \cite{HM}), or the $e$-property (see \cite{KPS,KSS}). 

For stochastic reaction-diffusion equations driven by space-time noise on bounded domains, the existence and uniqueness of invariant measures are established through proving the strong Feller property and irreducibility (see \cite{CE-1,PG-1,PG-3}). However, there exist no results on the exponential mixing partly due to the difficulty of utilizing the dissipativity conditions,  because of the lack of Ito formula/energy equalities.

On unbounded domains, several results are known for the existence of invariant measures; see \cite{AM03,CE24,EH01,MSY,TZ98}. 
However, so far there are no results on the uniqueness of invariant measures partly due to the fact that the strong Feller property collapses in this setting; see the result below. Furthermore, there exist no results at all on the exponential mixing of stochastic reaction-diffusion equations driven by multiplicative space-time noise even with strongly dissipative coefficients, partly because the solution of the equation is not a semi-martingale and the important tool, namely the It\^{o} formula, is not available. 
New ideas are therefore needed for SRDEs on unbounded domains. As far as we are aware of, this is the first paper to establish the ergodicity, exponential mixing and irreducibility  of stochastic reaction-diffusion equations driven by multiplicative space-time noise on unbounded domains. The results on exponential mixing are also new for (SRDEs) driven by multiplicative space-time noise on bounded domains. We also note that our proofs of exponential contractivity and exponential mixing can be adapted also to the weighted $L^p_{\lambda}$ space and to the Wasserstein metric $W_p$, for any $p\geq 1$.
\vskip 0.3cm
We now summarize the contributions of the paper.
\begin{itemize}
\item[I.]Non-strong Feller property: we show that stochastic reaction-diffusion equations on unbounded domains are not strong Feller, in contrast to the case of bounded domains.
\item[II.]Exponential contractivity: we obtain the exponential contractivity of solutions of SRDEs driven by space-time white noise with respect to the initial values under dissipativity assumptions, which could include coefficients with polynomial/exponential growth. 
Because It\^{o} formulae and energy equalities are not available, 
the existing methods in the literature fail, and it is tricky to make use of the dissipativity condition.  We overcome this shortcoming by finding a new way to fully exploit the comparison principles for stochastic partial differential  equations.
The results are also new for SRDEs driven by space-time white noise on bounded domains.

\item[III.]Exponential mixing:  we establish the exponential mixing property of stochastic reaction-diffusion equations driven by space-time white noise on unbounded domains, which could include coefficients with polynomial growth.  Here again, comparison principles of SRDEs play vital roles due to the lack of energy equations.

\item[IV.] Irreducibility: we establish the irreducibility for stochastic reaction-diffusion equations driven by space-time white noise on unbounded domains. To this end, we construct a specially designed controlled equation and find the appropriate controls. The Girsanov theorem for space-time white noise  plays an important role. The obstacle we need to overcome is that the control must belong to the Cameron-Martin space of the Brownian sheet,  whereas the solution itself lives only in a weighted $L^2_{\lambda}$ space. 
\end{itemize}
\vskip 0.5cm

The rest of the paper is organized as follows. In Section 2, we recall the precise framework and introduce the hypotheses. In Section 3, we show that stochastic reaction-diffusion equations(SRDEs) on unbounded domains are not strong Feller.
In Section 4, we prove the exponential contractivity of the solutions of SRDEs. In Section 5, we provide results on exponential mixing of SRDEs driven by space-time white noise on unbounded domains. In Section 6, we establish the irreducibility for SRDEs driven by space-time white noise on unbounded domains.

\vskip 0.3cm

\noindent\textbf{Convention on constants.} Throughout the paper, $C$ denotes a generic positive constant whose value may change from line to line. Other constants are denoted by $C_1$, $C_2$, etc. Their precise values are not important. Dependence on parameters will be indicated when needed, for example by $C_T$, $C_p$.
\newpage

\section{Framework}
In this section, we recall the definition of solutions and introduce the hypotheses.
\begin{definition}
We say that an adapted,
continuous random field $\{u(t,x): (t,x)\in [0,\infty)\times \mathbb{R}\}$ is a solution to the stochastic reaction-diffusion equation (SRDE) (\ref{3.1}) if, for every 
$\phi \in C_0^2(\mathbb{R})$ and $t\geq 0$,
\begin{align*}
	&\int_{\mathbb{R}} u(t,x)\phi(x) dx=\int_{\mathbb{R}} u_0(x)\phi(x) dx
	+\frac{1}{2}\int_{0}^{t} ds\int_{\mathbb{R}} u(s,x)\phi^{\prime\prime}(x)dx\nonumber\\
	&+\int_{0}^{t}ds\int_{\mathbb{R}} b(u(s,x))\phi(x)dx+ \int_{0}^{t}\int_{\mathbb{R}}\sigma(u(s,x))\phi(x)W(
	ds,dx),\quad \mathbb{P}\text{-a.s.}
\end{align*}
\end{definition}

It was shown in \cite{Wa} that $u$ is a solution to SRDE (\ref{3.1}) if and only if $u$ satisfies the following integral equation:
\begin{align}\label{2.1}
	u(t,x)=&P_tu_0(x)+\int_{0}^{t}\int_{\mathbb{R}} p_{t-s}(x,y)b(u(s,y))dyds\nonumber\\
	&+ \int_{0}^{t}\int_{\mathbb{R}} p_{t-s}(x,y)\sigma(u(s,y))W(ds,dy),\quad  \mathbb{P}\text{-a.s.}
\end{align}
where 
\begin{align}\label{260603.1931}
  p_{t}(x,y):=\frac{1}{\sqrt{2\pi t}} e^{-\frac{(x-y)^2}{2t}}, \quad P_t f(x) := \int_{\mathbb{R}} p_{t}(x,y) f(y)dy, \quad x,y\in\mathbb{R}.
\end{align}

%
%

\vskip 0.6cm
For $p\geq 1$ and $\lambda>0$, we introduce the weighted $L^p$ space:
\begin{align*}
  L^p_{\lambda} := \left\{f: \mathbb{R} \rightarrow \mathbb{R} \text{ measurable: } \int_{\mathbb{R}} |f(x)|^p e^{-\lambda|x|} dx <\infty  \right\}.
\end{align*}
Throughout the paper, we use the following heat kernel estimates
\begin{gather}
\label{260603.2012}  \int_{\mathbb{R}} p_t(x,y)^2 dy = \frac{1}{2\sqrt{\pi t}}, \quad \forall\  x\in\mathbb{R},\\
\label{260207.1152}
  \int_{\mathbb{R}} p_t(x,y) e^{-\lambda |x|} dx \leq e^{\frac{\lambda^2}{2}t} e^{-\lambda |y|}, \quad \forall\  y\in\mathbb{R},\\
  \label{260207.2020}
  \int_{\mathbb{R}} p_t(x,y)^2 e^{-\lambda |x|} dx \leq \frac{1}{2\sqrt{\pi t}}e^{\frac{\lambda^2}{4}t} e^{-\lambda |y|}, \quad \forall\  y\in\mathbb{R}.
\end{gather}
By H\"{o}lder's inequality and Fubini's theorem, we deduce from (\ref{260207.1152}) that
\begin{align}\label{260207.1918}
  \int_{\mathbb{R}} |P_t f(x)|^p e^{-\lambda |x|}dx & = \int_{\mathbb{R}} \left|\int_{\mathbb{R}} p_t (x,y) f(y) dy\right|^p e^{-\lambda |x|}dx \nonumber\\
  & \leq \int_{\mathbb{R}} \left(\int_{\mathbb{R}} p_t (x,y) |f(y)|^p dy \right) e^{-\lambda |x|}dx \nonumber\\
  & \leq e^{\frac{\lambda^2}{2}t} \int_{\mathbb{R}} |f(y)|^p e^{-\lambda |y|} dy .
\end{align}
In other words, for any $p\geq 1$,
\begin{align}\label{260208.1019}
  \Vert P_t f\Vert_{L^p_{\lambda}} \leq e^{\frac{\lambda^2}{2p}t} \Vert f\Vert_{L^p_{\lambda}}.
\end{align}
\vskip 0.4cm
We shall use the following  hypotheses in various places in the paper.
\vskip 0.4cm
\noindent {\bf (H1)} $b$ is Lipschitz, i.e.,  there exists a constant $L_{b}$ such that
\begin{align*}
|b(x) - b(y)| \leq L_{b} |x-y|, \quad\forall\  x,y\in\mathbb{R}.
\end{align*}
\noindent {\bf (H2)} $\sigma$ is Lipschitz, i.e.,  there exists a constant $L_{\sigma}$ such that
\begin{align*}
|\sigma(x) - \sigma(y)| \leq L_{\sigma} |x-y|, \quad\forall\  x,y\in\mathbb{R}.
\end{align*}
\noindent {\bf (H3)} Non-degeneracy: there exists a constant $c>0$ such that
\begin{align*}
|\sigma(x)|\geq c, \quad\forall\  x\in\mathbb{R}.
\end{align*}
\noindent {\bf (H4)} $b$ is dissipative, i.e. for some $\alpha>0$,
\begin{align*}
  (b(x) - b(y))(x-y) \leq -\alpha (x-y)^2, \quad\forall\  x,y\in\mathbb{R}.
\end{align*}

\vskip 0.6cm

\section{Non-strong Feller property}

Generally, solutions to stochastic reaction-diffusion on $\mathbb{R}$ are not strong Feller. This is true even in the case of additive-noise.
\begin{theorem}
  Let $u$ be the solution to the following equation driven by space-time white noise $W$:
\begin{align}\label{260209.1606}
\begin{cases}
    d u(t,x) = \frac{1}{2}\partial_{xx} u(t,x) dt  + W(dt,dx), \quad x\in\mathbb{R},\\
    u(0,x) = f(x), \quad x\in\mathbb{R},
\end{cases}
\end{align}
where $f\in L^p_{\lambda}$ for some $p\geq 1$. 
Then the solution $\{u(t): t\geq 0\}$ is not strong Feller on $L^p_{\lambda}$.
\end{theorem}

\begin{proof}
We prove only the case $p=2$; the proof for general $p\geq 1$ is similar.
For any $0<\mu<\lambda$, the space $L^2_{\mu}$ is in dense in $L^2_{\lambda}$.
If the initial value $f\in L^2_{\mu}$, then the solution $u^f \in C([0,\infty), L^2_{\mu})$ almost surely. Hence, the semigroup $T_t$ generated by the solution satisfies
\begin{align*}
  T_t \mathbf{1}_{L^2_{\mu}}(f) = \mathbb{P}(u^f(t)\in L^2_{\mu}) =1 .
\end{align*}
So if the solution $\{u(t): t\geq 0\}$ were strong Feller on $L^2_{\lambda}$, then we would have
\begin{align}\label{260208.1445}
   T_t \mathbf{1}_{L^2_{\mu}}(f) = 1, \quad \forall\  f\in L^2_{\lambda}.
\end{align}
However, we now show that there exists $f\in L^2_{\lambda} \setminus L^2_{\mu}$ such that
\begin{align}\label{260208.1448}
   T_t \mathbf{1}_{L^2_{\mu}}(f) = 0.
\end{align}
This contradicts (\ref{260208.1445}). Hence the solution $\{u(t): t\geq 0\}$ is not strong Feller on $L^2_{\lambda}$.

Take $f(x)=e^{\lambda_0 |x|}$, where $\lambda_0$ satisfies $\mu\leq 2\lambda_0 <\lambda$. Then $f\in L^2_{\lambda} \setminus L^2_{\mu}$. It remains to prove (\ref{260208.1448}). The solution to (\ref{260209.1606}) is
\begin{align*}
  u^f(t,x)  = P_t f(x) + \int_0^t\int_{\mathbb{R}} p_{t-s}(x,y) W(ds,dy).
\end{align*}
By the It\^{o} isometry, we see that
\begin{align*}
  \mathbb{E}\int_{\mathbb{R}} \left|\int_0^t\int_{\mathbb{R}} p_{t-s}(x,y) W(ds,dy)\right|^2 e^{-\mu|x|} dx = \mathbb{E}\int_0^t\int_{\mathbb{R}}\int_{\mathbb{R}}  p_{t-s}(x,y)^2  e^{-\mu|x|} dsdydx <\infty.
\end{align*}
However,  $P_t f\notin L^2_{\mu}$, since
\begin{align*}
  P_t f(x) = \mathbb{E}[e^{\lambda_0 |B_t + x|}] \geq \mathbb{E}[e^{\lambda_0 (|x| - |B_t|)}] = C_t e^{\lambda_0 |x|}.
\end{align*}
Thus, for every $t\geq 0$, $u^f(t) \notin L^2_{\mu}$ almost surely, which proves (\ref{260208.1448}).
\end{proof}
\section{Exponential contractivity}
Recall the hypotheses:
\vskip 0.4cm
\noindent {\bf (H2)} $\sigma$ is Lipschitz, i.e.,  there exists a constant $L_{\sigma}$ such that
\begin{align*}
|\sigma(x) - \sigma(y)| \leq L_{\sigma} |x-y|, \quad\forall\  x,y\in\mathbb{R}.
\end{align*}
\noindent {\bf (H4)} $b$ is dissipative, i.e. for some $\alpha>0$,
\begin{align*}
  (b(x) - b(y))(x-y) \leq -\alpha (x-y)^2, \quad\forall\  x,y\in\mathbb{R}.
\end{align*}

In this section, we assume that (\ref{3.1}) admits a unique solution. We show that the solutions to (\ref{3.1}) are exponentially contracting when the drift is  dissipative. We stress that the It\^{o} formula/energy equality is not available for 
such stochastic reaction-diffusion equations driven by space-time white noise. It is tricky to make use of the dissipativity condition (H4).
\begin{theorem}\label{260206.2029}
Let $u^i(t,\cdot)$ denotes the solution to euqation (\ref{3.1}) with initial value $u_0^i$, $i=1,2$. 
\begin{itemize}
  \item [(i)] Suppose that (H4) holds. 
Then for any $\lambda, t\geq 0$,
\begin{align}\label{260207.1931}
    \mathbb{E} \Vert u^2(t) - u^1(t)\Vert_{L^1_{\lambda}} \leq 2 e^{-(\alpha -\frac{\lambda^2}{2}) t} \mathbb{E}\Vert u_0^2 - u_0^1 \Vert_{L^1_{\lambda}} .
\end{align}
  \item [(ii)] Suppose that (H2) and (H4) hold. Assume that for some $\lambda\geq 0$,  $\mathbb{E}\int_{\mathbb{R}}|u^2(t,x) - u^1(t,x)|^2 e^{-\lambda|x|} dx <\infty$ for any $t\geq 0$. 
  If $\alpha>\frac{L_{\sigma}^4}{8}$ and $\kappa\geq 0$ satisfy
\begin{align}\label{260401.1108}
\kappa + \frac{\lambda^2}{8} < \alpha - \frac{L_{\sigma}^4}{8} \quad \text{ and } \quad   \kappa + \frac{\lambda^2}{4}\leq \alpha,
\end{align}
then there exists a constant $C_{\alpha,\lambda,\kappa}>0$ such that for any $t\geq 0$,
\begin{align}\label{260208.1116}
\left(\mathbb{E} \int_{\mathbb{R}} |u^2(t,x) - u^1(t,x)|^2 e^{-\lambda |x|} dx\right)^{\frac{1}{2}} \leq 2 C_{\alpha,\lambda,\kappa} e^{-\kappa t}\left(\mathbb{E}\Vert  u^2_0 - u^1_0 \Vert_{L^2_{\lambda}}^2\right)^{\frac{1}{2}}.
\end{align}
\end{itemize}
\end{theorem}

\begin{proof}

The proof is divided into two steps.

\textbf{Step 1.} We prove Theorem \ref{260206.2029} under the restriction that $u_0^2(x)\geq u^1_0(x)$ for all $x\in \mathbb{R}$. 

In this case, the comparison theorem gives $u^2(t,x)\geq u^1(t,x)$ for any $(t,x)\in\mathbb{R}_+\times\mathbb{R}$. 
The semigroup generated by $\frac{1}{2}\partial_{xx} - \alpha I$ is $e^{-\alpha t}P_t$, where $P_t$ is defined in (\ref{260603.1931}). 
Hence, the solution to (\ref{3.1}) admits the following mild form:
\begin{align*}
  u^i(t,x) = & e^{-\alpha t} P_t u_0^i(x) + \int_{0}^{t} \int_{\mathbb{R}} e^{-\alpha (t-s)} p_{t-s}(x,y)b_{\alpha}(u^i(s,y)) dsdy \nonumber\\
  & + \int_{0}^{t} \int_{\mathbb{R}} e^{-\alpha (t-s)} p_{t-s}(x,y)\sigma (u^i(s,y)) W(ds,dy),
\end{align*}
where $b_{\alpha}(u)=b(u)+\alpha u$. The crucial observation is that (H4) holds if and only if $b_{\alpha}(u)$ is decreasing. 
By subtraction, we have
\begin{align*}
  u^2(t,x) -u^1(t,x) = & e^{-\alpha t} P_t (u_0^2-u^1_0)(x) \nonumber\\
  & + \int_{0}^{t} \int_{\mathbb{R}} e^{-\alpha (t-s)} p_{t-s}(x,y)[b_{\alpha}(u^2(s,y)) - b_{\alpha}(u^1(s,y))] dsdy \nonumber\\
  & + \int_{0}^{t} \int_{\mathbb{R}} e^{-\alpha (t-s)} p_{t-s}(x,y)[\sigma (u^2(s,y)) - \sigma (u^1(s,y))]W(ds,dy).
\end{align*}
Since $u^2(s,y)\geq u^1(s,y)$ and $b_{\alpha}$ is decreasing, we deduce that
\begin{align}\label{260207.2002}
  u^2(t,x) - u^1(t,x) \leq & e^{-\alpha t} P_t (u_0^2-u^1_0)(x) \nonumber\\
  & + \int_{0}^{t} \int_{\mathbb{R}} e^{-\alpha (t-s)} p_{t-s}(x,y)[\sigma (u^2(s,y)) - \sigma (u^1(s,y))]W(ds,dy).
\end{align}
Taking expectations in the above inequality, we obtain
\begin{align*}
  \mathbb{E}|u^2(t,x) - u^1(t,x)| = \mathbb{E}[u^2(t,x) - u^1(t,x)] \leq e^{-\alpha t}\mathbb{E}[ P_t (u_0^2-u^1_0)(x)].
\end{align*}
Applying (\ref{260207.1918}) with $p=1$ and Fubini's theorem yields
\begin{align*}
   \mathbb{E}\Vert u^2(t) - u^1(t) \Vert_{L^1_{\lambda}} \leq e^{-(\alpha -\frac{\lambda^2}{2}) t} \mathbb{E}\Vert u_0^2 - u_0^1 \Vert_{L^1_{\lambda}} .
\end{align*}
This proves (\ref{260207.1931}) in the case $u_0^2\geq u^1_0$.

Next, we prove (\ref{260208.1116}) under the condition $u_0^2\geq u^1_0$. 
Taking the $L^2(\Omega)$-norm on both sides of (\ref{260207.2002}), we obtain 
\begin{align*}
  & \Vert u^2(t,x) -u^1(t,x) \Vert_{L^2(\Omega)} \nonumber\\
   \leq & e^{-\alpha t} \Vert P_t(u^2_0 - u^1_0)(x) \Vert_{L^2(\Omega)} + \left\Vert \int_{0}^{t}\int_{\mathbb{R}} e^{-\alpha (t-s)} p_{t-s}(x,y)[\sigma (u^2(s,y)) - \sigma (u^1(s,y))]W(ds,dy) \right\Vert_{L^2(\Omega)} \\
  \leq & e^{-\alpha t} \Vert P_t(u^2_0 - u^1_0)(x) \Vert_{L^2(\Omega)} + \left\{ \mathbb{E}\int_{0}^{t}\int_{\mathbb{R}} e^{-2\alpha (t-s)} p_{t-s}(x,y)^2 L_{\sigma}^2 | u^2(s,y) -  u^1(s,y)|^2 dsdy \right\}^{\frac{1}{2}} .
\end{align*}
We then take the $L^2_{\lambda}$-norm with respect to the spatial variable $x$ on both sides of the above inequality to get
\begin{equation*}
  \begin{aligned}
  & \left(\mathbb{E}\Vert u^2(t) - u^1(t)\Vert_{L^2_{\lambda}}^2\right)^{\frac{1}{2}} \leq  e^{-\alpha t} \left(\mathbb{E}\Vert P_t(u^2_0 - u^1_0)\Vert_{L^2_{\lambda}}^2\right)^{\frac{1}{2}} \\
 & + \left\{\mathbb{E}\int_{0}^{t}\int_{\mathbb{R}} \int_{\mathbb{R}} e^{-2\alpha (t-s)} p_{t-s}(x,y)^2 L_{\sigma}^2 \vert u^2(s,y) -  u^1(s,y)\vert^2 e^{-\lambda |x|} dsdydx \right\}^{\frac{1}{2}}.
\end{aligned}
\end{equation*}
By (\ref{260208.1019}), Fubini's theorem and (\ref{260207.2020}), we have
\begin{align}\label{260207.2034}
  & \left(\mathbb{E}\Vert u^2(t) - u^1(t)\Vert_{L^2_{\lambda}}^2\right)^{\frac{1}{2}} \leq  e^{-\alpha t} e^{\frac{\lambda^2 }{4}t}\left(\mathbb{E}\Vert u^2_0 - u^1_0\Vert_{L^2_{\lambda}}^2\right)^{\frac{1}{2}} \nonumber\\
 & + \left\{ \int_{0}^{t}  \frac{L_{\sigma}^2}{2\sqrt{\pi (t-s)}}e^{\frac{\lambda^2}{4}(t-s)}  e^{-2\alpha (t-s)} \mathbb{E}\int_{\mathbb{R}} \vert u^2(s,y) -  u^1(s,y)\vert^2 e^{-\lambda |y|} dyds \right\}^{\frac{1}{2}}.
\end{align}
For $\kappa\geq 0$, set
\[
\mathcal{N}_T^{\kappa}(u):= \sup_{0\leq t\leq T} \left[ e^{\kappa t} \left(\mathbb{E} \int_{\mathbb{R}} |u(t,x)|^2 e^{-\lambda |x|} dx\right)^{\frac{1}{2}} \right].
\]
It follows from (\ref{260207.2034}) that
\begin{align}\label{260208.1055}
  & \mathcal{N}_T^{\kappa}(u^2 -u^1) \nonumber\\
  \leq & \sup_{0\leq t\leq T} \left[e^{-(\alpha - \frac{\lambda^2}{4} -\kappa) t}\right] \left(\mathbb{E}\Vert  u^2_0 - u^1_0 \Vert_{L^2_{\lambda}}^2\right)^{\frac{1}{2}} \nonumber\\
  & + \sup_{0\leq t\leq T}\left\{\int_{0}^{t}  \frac{L_{\sigma}^2}{2\sqrt{\pi (t-s)}}e^{\frac{\lambda^2}{4}(t-s)}  e^{-2\alpha (t-s)} e^{2\kappa (t-s)} \mathcal{N}_T^{\kappa}(u^2 -u^1)^2 ds \right\}^{\frac{1}{2}} \nonumber\\
 \leq & \sup_{0\leq t\leq T} \left[e^{-(\alpha - \frac{\lambda^2}{4} -\kappa) t}\right] \left(\mathbb{E}\Vert  u^2_0 - u^1_0 \Vert_{L^2_{\lambda}}^2\right)^{\frac{1}{2}} \nonumber\\
  & + \mathcal{N}_T^{\kappa}(u^2 -u^1) \times \sup_{t\geq 0}\left\{ \frac{L_{\sigma}^2}{2\sqrt{\pi}}\int_{0}^{t}  \frac{1}{\sqrt{ (t-s)}}e^{- 2(\alpha-\kappa-\frac{\lambda^2}{8})(t-s)}  ds \right\}^{\frac{1}{2}} \nonumber\\
  \leq & \sup_{0\leq t\leq T} \left[e^{-(\alpha - \frac{\lambda^2}{4} -\kappa) t}\right] \left(\mathbb{E}\Vert  u^2_0 - u^1_0 \Vert_{L^2_{\lambda}}^2\right)^{\frac{1}{2}}  + \mathcal{N}_T^{\kappa}(u^2 -u^1) \times \left[8(\alpha-\frac{\lambda^2}{8}-\kappa)\right]^{-\frac{1}{4}} L_{\sigma}.
\end{align}
Since $\alpha> \frac{L_{\sigma}^4}{8}$ and (\ref{260401.1108}) hold, we have
\begin{align*}
\alpha  - \frac{\lambda^2}{4} -\kappa \geq 0 \quad\text{ and }\quad \left[8(\alpha-\frac{\lambda^2}{8}-\kappa)\right]^{-\frac{1}{4}} L_{\sigma} <1 .
\end{align*}
Hence (\ref{260208.1055}) implies that
\begin{align}\label{260605.1351}
  \mathcal{N}_T^{\kappa}(u^2 -u^1) \leq \left(\mathbb{E}\Vert  u^2_0 - u^1_0 \Vert_{L^2_{\lambda}}^2\right)^{1/2} +  \mathcal{N}_T^{\kappa}(u^2 -u^1) \left[8(\alpha-\frac{\lambda^2}{8}-\kappa)\right]^{-\frac{1}{4}} L_{\sigma}.
\end{align}
Squaring both sides of (\ref{260207.2034}), using the fact that $\mathbb{E}\Vert u^2(t) - u^1(t)\Vert_{L^2_{\lambda}}^2<\infty$ for any $t\geq 0$, and applying Gronwall's inequality, we find that $\mathcal{N}^{\kappa}_T(u^2 -u^1)<\infty$ for any $T>0$. Hence (\ref{260605.1351}) gives
\begin{align}
  \sup_{T\geq 0}\mathcal{N}_T^{\kappa}(u^2 -u^1) \leq  C_{\alpha,\kappa,\lambda} \left(\mathbb{E}\Vert  u^2_0 - u^1_0 \Vert_{L^2_{\lambda}}^2\right)^{\frac{1}{2}},
\end{align}
where 
\begin{align*}
    C_{\alpha,\lambda,\kappa}:= \left(1-\left[8\big(\alpha-\frac{\lambda^2}{8}-\kappa\big)\right]^{-\frac{1}{4}} L_{\sigma} \right)^{-1}.
\end{align*}
In particular, we obtain
\begin{align*}
\left(\mathbb{E} \int_{\mathbb{R}} |u^2(t,x) - u^1(t,x)|^2 e^{-\lambda |x|} dx\right)^{\frac{1}{2}} \leq C_{\alpha,\lambda,\kappa} e^{-\kappa t}\left(\mathbb{E}\Vert  u^2_0 - u^1_0 \Vert_{L^2_{\lambda}}^2\right)^{1/2}.
\end{align*}
This proves (\ref{260208.1116}) in the case of $u_0^2\geq u^1_0$.

\vskip 0.6cm

\textbf{Step 2.} We remove the restriction $u_0^2(x)\geq u_0^1(x)$, $x\in \mathbb{R}$.

Let $u^{(1,2)}(t,x)$ denote the solution to equation (\ref{3.1}) with initial value $u_0^1(x)\vee u_0^2(x)$. Note that the constants appearing in (\ref{260207.1931}) and (\ref{260208.1116}) are independent of the initial values. Hence, for $i=1,2$,
\begin{align*}
  \left(\mathbb{E} \int_{\mathbb{R}} |u^i(t,x) - u^{(1,2)}(t,x)|^2 e^{-\lambda |x|} dx\right)^{\frac{1}{2}} \leq C_{\alpha,\lambda,\kappa} e^{-\kappa t}\left(\mathbb{E}\Vert  u^i_0 - u^1_0\vee u^2_0 \Vert_{L^2_{\lambda}}^2\right)^{\frac{1}{2}}.
\end{align*}
Note that for $i=1,2$,
\begin{align*}
  |u_0^i(x) -(u_0^1\vee u_0^2)(x)| \leq |u_0^2(x) - u_0^1(x)|.
\end{align*}
Hence, by the triangle inequality we obtain
\begin{align*}
    \left(\mathbb{E} \int_{\mathbb{R}} |u^2(t,x) - u^{1}(t,x)|^2 e^{-\lambda |x|} dx\right)^{\frac{1}{2}} \leq 2 C_{\alpha,\lambda,\kappa} e^{-\kappa t}\left(\mathbb{E}\Vert  u^2_0 - u^1_0 \Vert_{L^2_{\lambda}}^2\right)^{\frac{1}{2}}.
\end{align*}
This completes the proof of (\ref{260208.1116}). The proof of  (\ref{260207.1931}) in this case is similar, and we omit the details.
\end{proof}

\section{Exponential mixing}
In this section, we first prove an abstract exponential mixing theorem for SRDEs. We then give concrete conditions under which the hypotheses of the abstract theorem are fulfilled.  
Recall the hypotheses:
\vskip 0.4cm
\noindent {\bf (H2)} $\sigma$ is Lipschitz, i.e.,  there exists a constant $L_{\sigma}$ such that
\begin{align*}
|\sigma(x) - \sigma(y)| \leq L_{\sigma} |x-y|, \quad\forall\  x,y\in\mathbb{R}.
\end{align*}
\noindent {\bf (H4)} $b$ is dissipative, i.e. for some $\alpha>0$,
\begin{align*}
  (b(x) - b(y))(x-y) \leq -\alpha (x-y)^2, \quad\forall\  x,y\in\mathbb{R}.
\end{align*}

Throughout this section, unless otherwise stated, we assume by default that (\ref{3.1}) admits a unique solution. 
\begin{theorem}\label{260209.1719}
Let $u$ be the solution of equation (\ref{3.1}) with initial value $f$. 
\begin{itemize}
  \item [(i)] Suppose that (H4) holds. If $\alpha > 0$, and for some $\lambda>0$ satisfying $\frac{\lambda^2}{2}<\alpha$, there exists $f\in L^1_{\lambda}$ such that
\begin{align}\label{260209.1121-1}
  \sup_{t\geq 0}\mathbb{E} \Vert u(t)\Vert_{L^1_{\lambda}} <\infty ,
\end{align}
then there exists a unique invariant measure in $L^1_{\lambda}$. Moreover, the solution is exponentially mixing with respect to the Wasserstein metric $W_1$.
  \item [(ii)] Suppose that (H2) and (H4) hold. If $\alpha > \frac{L_{\sigma}^4}{8}$, and for some $\lambda>0$ satisfying $\frac{\lambda^2}{8}<\alpha - \frac{L_{\sigma}^4}{8}$ and $\frac{\lambda^2}{4}<\alpha$, there exists $f\in L^2_{\lambda}$ such that
\begin{align}\label{260209.1121-2}
  \sup_{t\geq 0}\mathbb{E} [\Vert u(t)\Vert^2_{L^2_{\lambda}}] <\infty ,
\end{align}
then there exists a unique invariant measure in $L^2_{\lambda}$. Moreover, the solution is exponentially mixing with respect to the Wasserstein metric  $W_2$.
\end{itemize}
\end{theorem}

\begin{proof} We prove only (ii), since the proof of (i) is similar. 
Let $W(t,x)$ be the Brownian sheet corresponding to the space-time white noise $W(dt,dx)$, and let $W_1(t,x)$ be another Brownian sheet on $[0,\infty)\times\mathbb{R}$ that is independent of $W(t,x)$. Set
\begin{align*}
  \overline{W}(t,x) :=
  \begin{cases}
    W(t,x), & \mbox{if} \quad  t\geq 0,\  x\in\mathbb{R}, \\
    W_1(-t,x), & \mbox{if} \quad t<0, \ x\in\mathbb{R}.
  \end{cases}
\end{align*}
Let $\overline{\mathcal{F}}_t$ be the filtration satisfying usual conditions generated by $\{\overline{W}(s,x): s\leq t, \ x\in\mathbb{R}\}$. For any $\delta \geq 0$, consider the following SPDE:
\begin{align}\label{260209.1109}
\begin{cases}
    \partial_t u(t,x) = \frac{1}{2}\partial_{xx} u(t,x) + b(u(t,x)) +  \sigma(u(t,x)) \overline{W}(dt,dx), \quad x\in\mathbb{R},\\
    u(-\delta,x) = f(x), \quad x\in\mathbb{R}.
\end{cases}
\end{align}
Then (\ref{260209.1109}) admits a unique solution, 
which we denote by $u_{-\delta}(t,x)$ for $t\geq -\delta$.
By condition (\ref{260209.1121-2}) and the fact that 
the law $\mathcal{L}(u_{-\delta}(t))$ of $u_{-\delta}(t)$ equals the law of $u(t+\delta)$ on $L^2_{\lambda}$, 
%
there exists a constant $C>0$ independent of $t$ and $\delta$ such that
\begin{align}\label{260209.1547}
  \mathbb{E} [ \Vert u_{-\delta}(t)\Vert_{L^2_{\lambda}}^2 ] \leq C, \quad \forall\  \delta\geq 0 , \  t\geq -\delta .
\end{align}
For $\delta>\gamma$, consider the solutions $u_{-\delta}$ and $u_{-\gamma}$ on the interval $[-\gamma,\infty)$, Arguing as in the proof of Theorem \ref{260206.2029}, we obtain
\begin{equation}\label{260209.1618}
\begin{aligned}
  \left(\mathbb{E} [\Vert u_{-\delta}(t) - u_{-\gamma}(t)\Vert_{L^2_{\lambda}}^2 ]\right)^{\frac{1}{2}} \leq & 2 C_{\alpha,\lambda,\kappa} e^{-\kappa(t+\gamma)} \left(\mathbb{E}[\Vert  u_{-\delta}(-\gamma) - f\Vert_{L^2_{\lambda}}^2] \right)^{\frac{1}{2}} \\
  \leq &  2 C_{\alpha,\lambda,\kappa} e^{-\kappa(t+\gamma)} \left(\Vert f\Vert_{L^2_{\lambda}} + C \right), \quad t\geq -\gamma, 
\end{aligned}
\end{equation}
where we have used (\ref{260209.1547}) in the last line. This implies that for any $t\in\mathbb{R}$, the sequence of random variables $\{u_{-\gamma}(t)\}_{\gamma\geq 0}$ is a Cauchy sequence in $L^2(\Omega, L^2_{\lambda})$ as $\gamma\rightarrow\infty$. Let
\begin{align*}
  \xi := \lim_{\gamma\rightarrow\infty} u_{-\gamma}(0), \quad \mu := \mathcal{L}(\xi) \text{ on }  L^2_{\lambda}.
\end{align*}
Then $\mu$ is the unique invariant measure and independent of the initial value $f$. In fact, the law
\begin{align*}
 \mu_t := \mathcal{L} (u_0(t)) =  \mathcal{L} (u_{-t}(0))  \rightarrow \mu, \quad \text{ as } t\rightarrow\infty. 
\end{align*}
Hence $\mu$ is an invariant measure. The uniqueness of invariant measures follows from Theorem \ref{260206.2029}. Setting $t=0$ and $\gamma=s$ in (\ref{260209.1618}) and then letting $\delta\rightarrow +\infty$, we obtain
\begin{equation*}
\begin{aligned}
  W_2(\mu_s,\mu) \leq \left(\mathbb{E} [\Vert u_{-s}(0) - \xi \Vert_{L^2_{\lambda}}^2 ]\right)^{\frac{1}{2}}
  \leq  2 C_{\alpha,\lambda,\kappa} e^{-\kappa s} \left(\Vert f\Vert_{L^2_{\lambda}} + C \right) .
\end{aligned}
\end{equation*}
Therefore, the solution in $L^2_{\lambda}$ is exponentially mixing with respect to the Wasserstein metric $W_2$.
\end{proof}

\begin{proposition}
Assume that there exists a constant $\beta\in\mathbb{R}$ such that $\sigma(\beta) = 0$. 
\begin{itemize}
  \item [(i)] Suppose that (H4) holds. If $\alpha > 0$, then for any $\lambda>0$ satisfying $\frac{\lambda^2}{2}<\alpha$, there exists a unique invariant measure in $L^1_{\lambda}$. Moreover, the solution is exponentially mixing with respect to the Wasserstein metric $W_1$.
  \item [(ii)] Suppose that (H2) and (H4) hold. If $\alpha > \frac{L_{\sigma}^4}{8}$, then for any $\lambda>0$ satisfying $\frac{\lambda^2}{8}<\alpha - \frac{L_{\sigma}^4}{8}$ and $\frac{\lambda^2}{4}<\alpha$, there exists a unique invariant measure in $L^2_{\lambda}$. Moreover, the solution is exponentially mixing with respect to the Wasserstein metric  $W_2$.
\end{itemize}

\end{proposition}

\begin{proof} We only prove (i); the proof of (ii) is similar. By Theorem \ref{260209.1719}, it is enough to prove that for any $\lambda>0$ satisfying $\frac{\lambda^2}{2}<\alpha$, there exists $f\in L^1_{\lambda}$ such that the solution $u$ to (\ref{3.1}) with initial value $f$ satisfies
\begin{align}\label{260605.1713}
  \sup_{t\geq 0}\mathbb{E} \Vert u(t)\Vert_{L^1_{\lambda}} <\infty .
\end{align}

Taking $y=\beta$ in (H4) gives
\begin{align*}
  [b(x) - b(\beta) + \alpha (x-\beta)] (x-\beta)\leq 0 .
\end{align*}
Set
\begin{align*}
  \bar{b}(u) := b(u) - b(\beta), \quad \bar{b}_{\alpha}(u) := b(u) - b(\beta) +\alpha (u-\beta).
\end{align*}
Then
\begin{align}\label{260209.2104}
  \begin{cases}
    \bar{b}_{\alpha}(u)\leq 0, & \mbox{if } u\geq \beta, \\
    \bar{b}_{\alpha}(u)\geq 0, & \mbox{if } u\leq \beta.
  \end{cases}
\end{align}
Consider the following stochastic reaction-diffusion equation:
\begin{align*}
\begin{cases}
    \partial_t v(t,x) = \frac{1}{2}\partial_{xx} v(t,x) + \bar{b}(v(t,x)) + \sigma(v(t,x))W(dt,dx),\\
    v(0,x) \equiv \beta.
\end{cases}
\end{align*}
Since $\sigma(\beta) = 0$, the constant $\beta$ is the unique solution to this equation. We divide the remaining proof into two cases according to whether $b(\beta)\geq 0$ or $b(\beta)\leq 0$.

\vskip 0.6cm

\textbf{Case 1.} $b(\beta)\geq 0$. In this case, take any $\lambda>0$ satisfying $\frac{\lambda^2}{2}<\alpha$, and take any initial value 
\begin{align}\label{260605.1736}
  f\in \{h\in L^1_{\lambda}: h(x)\geq \beta \text{ for all } x\in\mathbb{R}\}.
\end{align}
%

We rewrite (\ref{3.1}) in the following form:
\begin{align*}
\begin{cases}
      \partial_t u(t,x) = \frac{1}{2}\partial_{xx} u(t,x) + \bar{b}(u(t,x)) + b(\beta) + \sigma(u(t,x))W(dt,dx), \\
    u(0,x) = f(x).
\end{cases}
\end{align*}
By the comparison principle, $u(t,x)\geq \beta$ for any $(t,x)\in\mathbb{R}_+\times\mathbb{R}$. As in (\ref{260207.2002}), using (\ref{260209.2104}), we obtain

\begin{equation*}
\begin{aligned}
  \beta\leq u(t,x) = & e^{-\alpha t} P_t f(x) + \int_{0}^{t} \int_{\mathbb{R}} e^{-\alpha (t-s)} p_{t-s}(x,y)[\bar{b}_{\alpha}(u(s,y))+b(\beta) + \alpha\beta] dsdy \nonumber\\
  & + \int_{0}^{t} \int_{\mathbb{R}} e^{-\alpha (t-s)} p_{t-s}(x,y)\sigma (u(s,y)) W(ds,dy) \\
  \leq & e^{-\alpha t} P_t f(x) + [ b(\beta) + \alpha\beta] \int_{0}^{t} e^{-\alpha (t-s)}  ds \nonumber\\
  & + \int_{0}^{t} \int_{\mathbb{R}} e^{-\alpha (t-s)} p_{t-s}(x,y)\sigma (u(s,y)) W(ds,dy) \\
    \leq & e^{-\alpha t} P_t f(x) + \frac{|b(\beta) +\alpha\beta|}{\alpha}  + \int_{0}^{t} \int_{\mathbb{R}} e^{-\alpha (t-s)} p_{t-s}(x,y)\sigma (u(s,y)) W(ds,dy).
\end{aligned}
\end{equation*}
Hence
\begin{equation*}
  \begin{aligned}
\mathbb{E}|u(t,x)| \leq & \mathbb{E}|u(t,x)-\beta| + |\beta| = \mathbb{E}[u(t,x)] -\beta + |\beta|\\
 \leq & e^{-\alpha t} P_t f(x) + \frac{|b(\beta)+\alpha\beta|}{\alpha} - \beta + |\beta|.
\end{aligned}
\end{equation*}
Applying (\ref{260207.1918}) with $p=1$ and Fubini's theorem yields
\begin{align*}
  \mathbb{E}\Vert u(t)\Vert_{L^1_{\lambda}} \leq e^{-(\alpha - \frac{\lambda^2}{2}) t} \Vert f\Vert_{L^1_{\lambda}} + \Big(\frac{|b(\beta)+\alpha\beta|}{\alpha} -\beta + |\beta|\Big)\frac{2}{\lambda}.
\end{align*}
This proves (\ref{260605.1713}) for all $f$ in (\ref{260605.1736}). 
%

\vskip 0.6cm

\textbf{Case 2.} $b(\beta)\leq 0$. In this case, take any $\lambda>0$ satisfying $\frac{\lambda^2}{2}<\alpha$, and take any initial value 
\begin{align}\label{260605.1726}
  f\in \{h\in L^1_{\lambda}: h(x)\leq \beta \text{ for all } x\in\mathbb{R}\}.
\end{align}
%
%
%
%
As in Case 1, we have
\begin{align*}
  \beta\geq u(t,x) \geq e^{-\alpha t} P_t f(x) - \frac{|b(\beta)+\alpha\beta|}{\alpha} + \int_{0}^{t} \int_{\mathbb{R}} e^{-\alpha (t-s)} p_{t-s}(x,y)\sigma (u(s,y)) W(ds,dy).
\end{align*}
Hence
\begin{equation*}
  \begin{aligned}
\mathbb{E}|u(t,x)| \leq & \mathbb{E}|\beta - u(t,x) | + |\beta| = \beta - \mathbb{E}[u(t,x)] + |\beta|\\
 \leq & e^{-\alpha t} |P_t f(x)| + \frac{|b(\beta)+\alpha\beta|}{\alpha} + \beta + |\beta|.
\end{aligned}
\end{equation*}
Therefore, (\ref{260605.1713}) also holds for all $f$ in (\ref{260605.1726}).
%
\end{proof}

\vskip 0.6cm

We next provide more concrete conditions on the coefficients $b$ and $\sigma$ which guarantee the existence and uniqueness of invariant measures and exponential mixing. 
Introduce

\vskip 0.4cm

\noindent {\bf (H5)} $b$ is locally Lipschitz and satisfies
  \begin{align}\label{260607.1313}
    |b(x) - b(y)| \leq L_b |x-y|(1+|x|^{\nu-1}+|y|^{\nu-1}),\quad x,y\in\mathbb{R},
  \end{align}
for some constant $L_b\geq 0$, where $\nu\geq 1$.

Under conditions (H2), (H4) and (H5), we shall prove the existence of a unique invariant measure in $L^2_{\lambda}$ and exponential mixing with respect to the Wasserstein metric $W_2$. As a preparation, we will give a  uniform-in-time moment estimate for the solution to (\ref{3.1}), which is also of independent interest. For this estimate, we impose the following conditions.
\begin{itemize}
  \item [(S1)] The initial value $u_0: \mathbb{R}\rightarrow \mathbb{R}$ is bounded and continuous.
  \item [(S2)] The coefficient $\sigma$ is Lipschitz and satisfies 
\begin{align}\label{260616.1301}
  |\sigma(u)|^2 \leq C_{\sigma} + G_{\sigma}|u|^2, \quad u\in\mathbb{R}.
\end{align}
  \item [(S3)] There exist constants $C_b\geq 0$ and $\theta\in\mathbb{R}$ such that
\begin{align}\label{260603.2134}
  ub(u)\leq C_b - \theta|u|^2, \quad \forall\  u\in\mathbb{R}.
\end{align}
\end{itemize}

\begin{proposition}\label{260604.1227}
Assume that (S1)--(S3) and (H5) hold, and that
\[
\theta> \frac{G_{\sigma}^2}{8}.
\]
Then the unique solution $u$ to equation (\ref{3.1}) satisfies
\begin{align}
  \sup_{t\geq 0}\sup_{x\in\mathbb{R}} \mathbb{E}|u(t,x)|^2 <\infty.
\end{align}
In particular, for any $\lambda>0$,
\begin{align*}
  \sup_{t\geq 0} \mathbb{E}\int_{\mathbb{R}} |u(t,x)|^2 e^{-\lambda |x|} dx \leq  \frac{2}{\lambda}\sup_{t\geq 0}\sup_{x\in\mathbb{R}}  \mathbb{E}|u(t,x)|^2 <\infty .
\end{align*}
\end{proposition}

\begin{remark}
Theorem 3 of \cite{AM03} also establishes a related uniform-in-time estimate in a weighted $L^p$ space. However, the uniform-in-time estimate is obtained in a different way here. As a result, our condition on $\theta$ is expressed in terms of the linear-growth constant $G_{\sigma}$ in (S2), rather than the Lipschitz constant of $\sigma$, as in \cite{AM03}. In the present setting, the Lipschitz continuity of $\sigma$ is used mainly to ensure well-posedness of the equation, whereas the a priori estimate itself only uses the growth bound (\ref{260616.1301}). In this sense, the condition imposed on $\theta$ here is less restrictive than that in \cite{AM03}.

\end{remark}

The proof of Proposition \ref{260604.1227} is inspired by that of Theorem 3 in \cite{AM03} and will be given at the end of this section. 

\vskip 0.6cm

\begin{theorem}\label{260604.1328}

Suppose that (H2), (H4) and (H5) hold. 
If $\alpha > \frac{L_{\sigma}^4}{8}$, 
then for any $\lambda>0$ satisfying $\frac{\lambda^2}{8}<\alpha - \frac{L_{\sigma}^4}{8}$ and $\frac{\lambda^2}{4}<\alpha$, we have
\begin{align}\label{5.2-0}
  \sup_{t\geq 0}\mathbb{E} [\Vert u(t)\Vert^2_{L^2_{\lambda}}] <\infty ,
\end{align}
for every bounded continuous initial value $f$ on $\mathbb{R}$. 
In particular, there exists a unique invariant measure in $L^2_{\lambda}$. Moreover, the solution is exponentially mixing with respect to the Wasserstein metric  $W_2$.

\end{theorem}

\begin{remark}
If the initial value of (\ref{3.1}) is bounded and continuous function, then (H2), (H4) and (H5) imply, by Theorem 1 of \cite{AM03} (see also Theorem 3.4.1 of \cite{MZ99}), that (\ref{3.1}) admits a unique solution.
\end{remark}

\begin{proof}
By Theorem \ref{260209.1719}, 
we only need to prove (\ref{5.2-0}) for any bounded continuous initial value $f$ on $\mathbb{R}$. 

Since $\alpha> \frac{L_{\sigma}^4}{8}$, there exist $\varepsilon, \eta>0$ such that
\[
\alpha-\varepsilon> \frac{(L_{\sigma}^2+\eta)^2}{8}.
\]
Fix such $\varepsilon$ and $\eta$. Taking $y=0$ in (H4) gives
\begin{align*}
  (b(u) - b(0))u \leq -\alpha |u|^2, \quad\forall\  u\in\mathbb{R},
\end{align*}
which implies
\begin{align}\label{260605.2001}
  ub(u) \leq -(\alpha - \varepsilon) |u|^2 + \frac{|b(0)|^2}{4\varepsilon}, \quad\forall\  u\in\mathbb{R}.
\end{align}
By (H2), $|\sigma(u) | \leq |\sigma(0)| + L_{\sigma} |u|$. Hence
\begin{align}\label{260605.2002}
  |\sigma(u)|^2 \leq |\sigma(0)|^2 + L_{\sigma}^2 |u|^2 + 2 L_{\sigma} |\sigma(0)||u| \leq (L_{\sigma}^2 + \eta)|u|^2 + \left(1+ \frac{L_{\sigma}^2}{\eta}\right)|\sigma(0)|^2 .
\end{align}
Set $\theta:= \alpha -\varepsilon$ and $G_{\sigma}:= L_{\sigma}^2+\eta$. Then Proposition \ref{260604.1227} gives
\[
  \sup_{t\geq 0} \mathbb{E}\int_{\mathbb{R}} |u(t,x)|^2 e^{-\lambda |x|} dx<\infty,
\]
which proves (\ref{5.2-0}) for any $\lambda>0$. 

\end{proof}

Next, we turn to the proof of Proposition \ref{260604.1227}. 
We begin with three auxiliary lemmas. Consider the auxiliary equation
%
\begin{align}\label{260603.1828}
	\left\{
	\begin{aligned}
		d v(t,x) &=[\frac{1}{2}\partial_{xx} v(t,x) - mv(t,x)]dt + \tilde{b}(v(t,x))dt + \sigma(v(t,x)) W(dt,dx),\\
		v(0,x)&=u_0(x), \quad x\in \mathbb{R},
	\end{aligned}
	\right.
\end{align}
where $m>0$.

\begin{lemma}\label{260603.2255}
Assume that (S1) and (S2) hold. Let $\tilde{v}$ be the unique solution to (\ref{260603.1828}) with $\tilde{b}(v)  = - \gamma v$. 
Then, for any $m>0$, $p\geq 2$ and $\gamma > 2p^2 G_{\sigma}^2$, we have
\begin{align}\label{260603.1925}
  \sup_{t\geq 0}\sup_{x\in\mathbb{R}} \mathbb{E} [ |\tilde{v}(t,x)|^p] <\infty.
\end{align}
\end{lemma}

\begin{proof}
The semigroup generated by $\frac{1}{2}\partial_{xx} - mI - \gamma I$ is $e^{-(m+\gamma) t}P_t$, where $P_t$ is defined in (\ref{260603.1931}). Hence, $\tilde{v}$ admits the following mild representation:
\begin{align*}
  \tilde{v}(t,x) = e^{-(m+\gamma)t}P_t u_0(x) + \int_{0}^{t}\int_{\mathbb{R}} e^{-(m+\gamma)(t-s)}p_{t-s}(x,y)\sigma(\tilde{v}(s,y)) W(ds,dy).
\end{align*}
Set
\[
\mathcal{N}_T(\tilde{v}):= \sup_{0\leq t\leq T}\sup_{x\in\mathbb{R}}\Vert \tilde{v}(t,x)\Vert_{L^p(\Omega)}.
\]
Then
\begin{align}\label{260603.2049}
  \mathcal{N}_T(\tilde{v}) \leq & \sup_{0\leq t\leq T}\sup_{x\in\mathbb{R}}\big[ e^{-(m+\gamma)t}|P_t u_0(x)| \big] \nonumber\\
  & + \sup_{0\leq t\leq T}\sup_{x\in\mathbb{R}}\left\Vert \int_{0}^{t}\int_{\mathbb{R}} e^{-(m+\gamma)(t-s)}p_{t-s}(x,y)\sigma(\tilde{v}(s,y)) W(ds,dy) \right\Vert_{L^p(\Omega)} \nonumber\\
  =: & I(T) + II(T).
\end{align}
By the contractivity of $P_t$, we have
\begin{align}\label{260603.2156}
  I(\infty)\leq \sup_{x\in\mathbb{R}} |u_0(x)|.
\end{align}  
The term $II$ can be estimated as follows.
\begin{align}\label{260604.1025}
  II(T)^2 = & \sup_{0\leq t\leq T}\sup_{x\in\mathbb{R}} \left\{\mathbb{E}\left| \int_{0}^{t}\int_{\mathbb{R}} e^{-(m+\gamma)(t-s)}p_{t-s}(x,y)\sigma(\tilde{v}(s,y)) W(ds,dy) \right|^p\right\}^{\frac{2}{p}}  \nonumber\\
  \leq & \sup_{0\leq t\leq T}\sup_{x\in\mathbb{R}} \left\{ (4p)^{\frac{p}{2}} \mathbb{E}\left( \int_{0}^{t}\int_{\mathbb{R}} e^{-2(m+\gamma)(t-s)}p_{t-s}(x,y)^2\sigma(\tilde{v}(s,y))^2 dsdy \right)^{\frac{p}{2}}\right\}^{\frac{2}{p}}  \nonumber\\
  \leq & 4p \sup_{0\leq t\leq T}\sup_{x\in\mathbb{R}} \int_{0}^{t}\int_{\mathbb{R}} e^{-2(m+\gamma)(t-s)}p_{t-s}(x,y)^2 \Vert \sigma(\tilde{v}(s,y))^2 \Vert_{L^{\frac{p}{2}}(\Omega)}dsdy \nonumber\\
  \leq & 4p \sup_{0\leq t\leq T}\sup_{x\in\mathbb{R}} \int_{0}^{t} e^{-2(m+\gamma)(t-s)} \left(\int_{\mathbb{R}} p_{t-s}(x,y)^2 dy\right) \sup_{y\in\mathbb{R}}\Vert \sigma(\tilde{v}(s,y))^2 \Vert_{L^{\frac{p}{2}}(\Omega)}ds  \nonumber\\
  \leq & 4p \sup_{0\leq t\leq T} \int_{0}^{t} e^{-2(m+\gamma)(t-s)} \frac{1}{2\sqrt{\pi (t-s)}} \sup_{y\in\mathbb{R}}\Vert C_{\sigma} + G_{\sigma}|\tilde{v}(s,y)|^2 \Vert_{L^{\frac{p}{2}}(\Omega)}ds  \nonumber\\
  \leq & \frac{2pC_{\sigma}}{\sqrt{\pi }}  \sup_{0\leq t\leq T} \int_{0}^{t} e^{-2(m+\gamma)(t-s)} \frac{1}{\sqrt{t-s}} ds \nonumber\\
   & + \frac{2pG_{\sigma}}{\sqrt{\pi }} \sup_{0\leq t\leq T} \int_{0}^{t} e^{-2(m+\gamma)(t-s)} \frac{1}{\sqrt{ t-s}} \sup_{y\in\mathbb{R}}\Vert \tilde{v}(s,y) \Vert^2_{L^p(\Omega)}ds  \nonumber\\
   \leq & \frac{2pC_{\sigma}}{\sqrt{\pi }}  \int_{0}^{\infty} e^{-2(m+\gamma)s} \frac{1}{\sqrt{s}} ds \nonumber\\
   & + \frac{2pG_{\sigma}}{\sqrt{\pi }} \int_{0}^{\infty} e^{-2(m+\gamma)s} \frac{1}{\sqrt{ s}} ds \times\mathcal{N}_T(\tilde{v})^2,
\end{align}
where we used the Burkholder-Davis-Gundy inequality (see Theorem B.1 in \cite{K14}) and (\ref{260603.2012}). Set
\begin{align}\label{260604.1050}
  \widetilde{C}:= \frac{2p}{\sqrt{\pi }} \int_{0}^{\infty} e^{-2(m+\gamma)s} \frac{1}{\sqrt{ s}} ds = \frac{2p}{\sqrt{2(m+\gamma) }}.
\end{align}
%
By (\ref{260604.1025}), 
\begin{align}\label{260603.2044}
  II(T)\leq & (\widetilde{C} C_{\sigma})^{\frac{1}{2}}
    + (\widetilde{C}G_{\sigma})^{\frac{1}{2}} \times\mathcal{N}_T(\tilde{v}).
\end{align}
%
Combining (\ref{260603.2049}), (\ref{260603.2156}) and (\ref{260603.2044}) yields
\begin{align}\label{260605.1228}
  \mathcal{N}_T(\tilde{v}) \leq \sup_{x\in\mathbb{R}}|u_0(x)| + (\widetilde{C}C_{\sigma})^{\frac{1}{2}}
    + (\widetilde{C}G_{\sigma})^{\frac{1}{2}} \times\mathcal{N}_T(\tilde{v}).
\end{align}
By (\ref{260604.1050}), the assumption $\gamma > 2p^2 G_{\sigma}^2$ implies $\widetilde{C}G_{\sigma}<1$. Since $\sigma$ is Lipschitz, it is well known that the solution $\tilde{v}$ to (\ref{260603.1828}) satisfies $\mathcal{N}_T(\tilde{v})<\infty$ for any $T>0$. 
Therefore, (\ref{260605.1228}) gives
\[
\sup_{T\geq 0}\mathcal{N}_T(\tilde{v})<\infty,
\]
which proves (\ref{260603.1925}).
\end{proof}

Recall condition (S3), and set 
\begin{align}\label{260607.1423}
  C_{b,\theta}:= C_b + \theta + \sup_{|x|\leq 1}|b(x)|. 
\end{align}
It follows from (\ref{260603.2134}) that
\begin{align}\label{260604.1239}
  \begin{cases}
    b(u)\leq -\theta u + C_{b,\theta}, & \forall\ u\geq 0, \\
    b(u)\geq -\theta u - C_{b,\theta}, & \forall\ u\leq 0.
  \end{cases}
\end{align}

\begin{lemma}\label{260604.1225}
 Assume that (S1)--(S3) and (H5) hold. Define
  \begin{align*}
    h_{+}(u):=
  \begin{cases}
    C_{b,\theta}, & u\geq 0,\\
    -\gamma u + \sup_{u\leq v\leq 0}\big(b(v) - b(0) + C_{b,\theta}\big), & u\leq 0,
  \end{cases}
  \end{align*}
where $\gamma > 2(2\nu)^2 G_{\sigma}^2$. Let $v_{+}$ be the unique solution to equation (\ref{260603.1828}) with $\tilde{b}=h_{+}$. If
\[
m >\frac{G_{\sigma}^2}{8},
\]
then
\begin{align}\label{260603.2314}
  \sup_{t\geq 0}\sup_{x\in\mathbb{R}}\mathbb{E}|v_{+}(t,x)|^2 <\infty.
\end{align}

\end{lemma}

\begin{proof}
By (H5), $h_+$ also satisfies the polynomial local Lipschitz condition (H5) as $b$. Moreover,
\[
uh_{+}(u) \leq \frac{1}{2}(|u|^2 + C_{b,\theta}^2).
\] 
Therefore, Theorem 1 of \cite{AM03} (see also Theorem 3.4.1 of \cite{MZ99}) applies to (\ref{260603.1828}) with $\tilde{b}=h_{+}$, and yields the existence and uniqueness of the solution $v_{+}$. It remains to prove the uniform estimate (\ref{260603.2314}). 

Consider an auxiliary equation
\begin{align}\label{260605.1008}
	\left\{
	\begin{aligned}
		d v(t,x) &=[\frac{1}{2}\partial_{xx} v(t,x) - mv(t,x)]dt + h_+(v_+(t,x))dt + \sigma(v(t,x)) W(dt,dx),\\
		v(0,x)&=u_0(x), \quad x\in \mathbb{R}.
	\end{aligned}
	\right.
\end{align}
Since $\sigma$ is Lipschitz and the drift term $h_+(v_+)$ does not depend on $v$, the above equation has a unique solution. Meanwhile, $v_+$ also satisfies the above equation, which implies that the unique solution to (\ref{260605.1008}) is exactly $v_+$. 
Now construct a sequence of approximating solution $v_+^n$. Set $v_+^0(t) = u_0$ for any $t\geq 0$, and then iteratively define
\begin{align*}
  v_+^{n+1}(t,x) = & e^{-mt}P_t u_0(x) + \int_{0}^{t}\int_{\mathbb{R}} e^{-m(t-s)}p_{t-s}(x,y)h_+(v_+(s,y))dsdy \nonumber\\
& + \int_{0}^{t}\int_{\mathbb{R}} e^{-m(t-s)}p_{t-s}(x,y)\sigma(v_+^n(s,y))W(ds,dy).
\end{align*}
Then 
\begin{align}\label{260605.1117}
  \lim_{n\rightarrow\infty} v_+^{n}(t,x) = v_+(t,x) \quad\text{ in } L^2(\Omega),\quad \forall\ (t,x)\in [0,\infty)\times\mathbb{R}.
\end{align}
See, e.g., Chapter 3 of Walsh \cite{Wa}. 
%
%

Set
\[
\mathcal{N}(v_+^n):= \sup_{t\geq 0}\sup_{x\in\mathbb{R}}\Vert v_+^n(t,x)\Vert_{L^2(\Omega)}.
\]
Then
\begin{align}\label{260604.1046}
  \mathcal{N}(v_+^{n+1}) \leq & \sup_{t\geq 0}\sup_{x\in\mathbb{R}}\big[ e^{-mt}|P_t u_0(x)| \big] \nonumber\\
  & + \sup_{t\geq 0}\sup_{x\in\mathbb{R}}\left\Vert \int_{0}^{t}\int_{\mathbb{R}} e^{-m(t-s)}p_{t-s}(x,y) h_+(v_+(s,y)) dsdy \right\Vert_{L^2(\Omega)} \nonumber\\
  & + \sup_{t\geq 0}\sup_{x\in\mathbb{R}}\left\Vert \int_{0}^{t}\int_{\mathbb{R}} e^{-m(t-s)}p_{t-s}(x,y)\sigma(v_{+}^n(s,y)) W(ds,dy) \right\Vert_{L^2(\Omega)} \nonumber\\
  =: & I_1 + I_2 + I_3.
\end{align}

We first prove that $I_2<\infty$ by using the comparison principle. 
\begin{align}\label{260603.2248}
  I_2\leq & \sup_{t\geq 0}\sup_{x\in\mathbb{R}} \int_{0}^{t}\int_{\mathbb{R}} e^{-m(t-s)}p_{t-s}(x,y) \Vert h_+(v_+(s,y))\Vert_{L^2(\Omega)} dsdy \nonumber\\
  \leq &\sup_{t\geq 0}\int_{0}^{t} e^{-m(t-s)} \sup_{y\in\mathbb{R}}\Vert h_+(v_+(s,y))\Vert_{L^2(\Omega)} ds.
\end{align}
By the definition of $h_+$ and (H5), 
\begin{align}
\label{260603.2241} |h_+(u)|\leq & C_{b,\theta} + (\gamma + L_b)|u| + L_b |u|^{\nu}, \quad \forall\ u\in\mathbb{R},\\
\label{260603.2242}  h_+(u) \geq & -\gamma u,\quad \forall\ u\in\mathbb{R}.
\end{align}
It follows from (\ref{260603.2242}) and the comparison principle that
\[
v_+(t,x) \geq \tilde{v}(t,x), \quad \forall\ (t,x)\in [0,\infty)\times\mathbb{R}.
\]
Since $h_+$ is nonincreasing and takes values in $[C_{b,\theta},\infty)$, (\ref{260603.2241}) gives
\begin{align}\label{260603.2247}
  |h_+(v_+(s,y))| \leq |h_+(\tilde{v}(s,y))| \leq C_{b,\theta} + (\gamma + L_b)|\tilde{v}(s,y)| + L_b |\tilde{v}(s,y)|^{\nu}.
\end{align}
Substituting (\ref{260603.2247}) into (\ref{260603.2248}) yields
\begin{align*}
  I_2\leq & \sup_{t\geq 0}\int_{0}^{t} e^{-m(t-s)} \sup_{y\in\mathbb{R}}\Vert C_{b,\theta} + (\gamma + L_b)|\tilde{v}(s,y)| + L_b |\tilde{v}(s,y)|^{\nu}\Vert_{L^2(\Omega)} ds  \\
  \leq & C_{b,\theta}\sup_{t\geq 0}\int_{0}^{t} e^{-m(t-s)}  ds + (\gamma + L_b) \sup_{t\geq 0}\int_{0}^{t} e^{-m(t-s)} \sup_{y\in\mathbb{R}}\Vert \tilde{v}(s,y) \Vert_{L^2(\Omega)} ds  \\
  & + L_b \sup_{t\geq 0}\int_{0}^{t} e^{-m(t-s)} \sup_{y\in\mathbb{R}}\Vert  \tilde{v}(s,y)\Vert^{\nu}_{L^{2\nu}(\Omega)} ds\nonumber\\
  < & \infty,
\end{align*}
since Lemma \ref{260603.2255} gives 
\[
\sup_{s\geq 0}\sup_{y\in\mathbb{R}} \big[ \Vert \tilde{v}(s,y) \Vert_{L^{2}(\Omega)} + \Vert \tilde{v}(s,y) \Vert_{L^{2\nu}(\Omega)}\big]<\infty.
\]

The terms $I_1$ and $I_3$ can be estimated as in (\ref{260603.2156}) and (\ref{260603.2044}). This gives
\begin{gather}
\label{260603.2311}  I_1 \leq \sup_{x\in\mathbb{R}}|u_0(x)|, \\
\label{260603.2306}  I_3  \leq (C_+ C_{\sigma})^{\frac{1}{2}} + (C_+ G_{\sigma})^{\frac{1}{2}} \mathcal{N}(v_+^n),
\end{gather}
where
\[
C_+ := \frac{1}{2\sqrt{2m}}
\]
is obtained by using the It\^{o} isometry instead of the Burkholder-Davis-Gundy inequality. Combining (\ref{260604.1046}), (\ref{260603.2311}) and (\ref{260603.2306}) yields
\begin{align}\label{260605.1114}
  \mathcal{N}(v_+^{n+1}) \leq \sup_{x\in\mathbb{R}}|u_0(x)| + I_2 + (C_+ C_{\sigma})^{\frac{1}{2}} + (C_+ G_{\sigma})^{\frac{1}{2}} \mathcal{N}(v_+^n).
\end{align}
By (\ref{260603.2012}), 
\begin{align*}
  \mathcal{N}(v_+^{1}) 
  \leq \sup_{x\in\mathbb{R}}|u_0(x)| + I_2  + \sup_{y\in\mathbb{R}}|\sigma(u_0(y))|\times\left(\int_0^{\infty}e^{-2ms}\frac{1}{2\sqrt{\pi s}}ds\right)^{\frac{1}{2}}.
\end{align*}
Since $u_0$ is bounded and $I_2<\infty$, the preceding estimate gives $\mathcal{N}(v_+^{1})<\infty$. 
Since $m > \frac{G_{\sigma}^2}{8} $ is equivalent to $C_+ G_{\sigma}<1$, iterating (\ref{260605.1114}) yields
\[
\sup_{n\geq 1}\mathcal{N}(v_+^n)<\infty.
\]
By Fatou's lemma and (\ref{260605.1117}),
\[
\mathcal{N}(v_+) \leq \liminf_{n\rightarrow\infty} \mathcal{N}(v_+^n)<\infty,
\]
which proves (\ref{260603.2314}).
\end{proof}

\begin{lemma}\label{260604.1327}
Assume that (S1)--(S3) and (H5) hold. Define
  \begin{align*}
    h_{-}(u):=
  \begin{cases}
        -\gamma u + \inf_{0\leq v\leq u}\big(b(v) - b(0) - C_{b,\theta}\big), & u\geq 0, \\
        - C_{b,\theta}, & u\leq 0,\\
  \end{cases}
  \end{align*}
where $\gamma > 2(2\nu)^2 G_{\sigma}^2$. Let $v_{-}$ be the unique solution to (\ref{260603.1828}) with $\tilde{b}=h_{-}$. If
\[
m >\frac{G_{\sigma}^2}{8},
\]
then
\begin{align*}
  \sup_{t\geq 0}\sup_{x\in\mathbb{R}}\mathbb{E}|v_{-}(t,x)|^2 <\infty.
\end{align*}

\end{lemma}

\begin{proof}
  The proof is the same as that of Lemma \ref{260604.1225}, except for the estimate of $I_2$. We give the necessary modifications.
  
By the definition of $h_{-}$ and condition (H5), 
\begin{align}
\label{260604.1216} |h_{-}(u)|\leq & C_{b,\theta} + (\gamma + L_b)|u| + L_b |u|^{\nu}, \quad \forall\ u\in\mathbb{R},\\
\label{260604.1215}  h_{-}(u) \leq & -\gamma u,\quad \forall\ u\in\mathbb{R}.
\end{align}
It follows from (\ref{260604.1215}) and the comparison principle that
\[
v_{-}(t,x) \leq \tilde{v}(t,x), \quad \forall\ (t,x)\in [0,\infty)\times\mathbb{R}.
\]
Since $h_-$ is nonincreasing and takes values in $(-\infty, -C_{b,\theta}]$, by (\ref{260604.1216}) we still get
\[
  |h_{-}(v_{-}(s,y))| \leq |h_{-}(\tilde{v}(s,y))| \leq C_{b,\theta} + (\gamma + L_b)|\tilde{v}(s,y)| + L_b |\tilde{v}(s,y)|^{\nu}.
\]
\end{proof}

\begin{proof}[Proof of Proposition \ref{260604.1227}]
By Theorem 1 of \cite{AM03} (see also Theorem 3.4.1 of \cite{MZ99}), there exists a unique solution $u$ to (\ref{3.1}) under conditions (S1)--(S3) and (H5).  
Let $b_{\theta}(u) = b(u)+\theta u$. 
Then $u$ also solves (\ref{260603.1828}) with $m=\theta>\frac{G_{\sigma}^2}{8}$ and $\tilde{b}(u)=b_{\theta}(u)$. By (\ref{260604.1239}) and the definitions of $h_+$ and $h_{-}$, a direct computation gives
\begin{align*}
  h_{-}(u) \leq b_{\theta}(u)\leq h_+(u), \quad \forall \ u\in\mathbb{R}.
\end{align*}
The comparison principle then yields
\[
v_{-}(t,x) \leq u(t,x) \leq v_{+}(t,x), \quad \forall \ (t,x)\in [0,\infty)\times\mathbb{R}.
\]
Therefore, by Lemma \ref{260604.1225} and \ref{260604.1327}, 
\[
\sup_{t\geq 0}\sup_{x\in\mathbb{R}} \mathbb{E} |u(t,x)|^2 \leq \sup_{t\geq 0}\sup_{x\in\mathbb{R}} \mathbb{E} |v_+(t,x)|^2 + \sup_{t\geq 0}\sup_{x\in\mathbb{R}} \mathbb{E} |v_{-}(t,x)|^2 <\infty,
\]
which completes the proof of Proposition \ref{260604.1227}.
\end{proof}

\section{Irreducibility}
In this section, we prove the irreducibility for the stochastic reaction-diffusion equation (\ref{3.1}). We recall the assumptions used in this section.

\vskip 0.4cm

\noindent {\bf (H1)} $b$ is Lipschitz, i.e.,  there exists a constant $L_{b}$ such that
\begin{align*}
|b(x) - b(y)| \leq L_{b} |x-y|, \quad\forall\  x,y\in\mathbb{R}.
\end{align*}
\noindent {\bf (H2)} $\sigma$ is Lipschitz, i.e.,  there exists a constant $L_{\sigma}$ such that
\begin{align*}
|\sigma(x) - \sigma(y)| \leq L_{\sigma} |x-y|, \quad\forall\  x,y\in\mathbb{R}.
\end{align*}
\noindent {\bf (H3)} $\sigma$ is non-degenerate, i.e. there is a constant $c>0$ such that
\begin{align*}
  |\sigma(x)|\geq c>0, \quad\forall\  x\in\mathbb{R}.
\end{align*}
\begin{theorem}\label{260312.1046}
Under (H1)--(H3), the solution to equation (\ref{3.1}) is irreducible in $L^2_{\lambda}$, for every $\lambda>0$. More precisely, for any initial value $f\in L^2_{\lambda}$, any $T,r>0$ and $g\in L^2_{\lambda}$,
\begin{align}\label{260210.1127}
  \mathbb{P}(\Vert u(T) - g\Vert_{L^2_{\lambda}}<r)>0,
\end{align}
where $u(t,x)$ is the solution to (\ref{3.1}) with initial value $f$.
\end{theorem}

\begin{proof}

Recall that the solution $u(t,x)$ admits the mild form
\begin{align*}
	u(t,x)=&P_tf(x)+\int_{0}^{t}\int_{\mathbb{R}} p_{t-s}(x,y)b(u(s,y))dyds\nonumber\\
	&+ \int_{0}^{t}\int_{\mathbb{R}} p_{t-s}(x,y)\sigma(u(s,y))W(ds,dy),\quad \mathbb{P}\text{-a.s.}
\end{align*}
Since the coefficients $b$ and $\sigma$ are Lipschitz, the above equation is well-posed in the space $C([0,\infty), L^2_{\lambda})$. By Lemma 3.3 in \cite{LSZ25} and Gronwall's inequality, it follows that for any $T,\lambda>0$ and $p\geq 2$, 
\begin{align}\label{260313.1659}
  \mathbb{E} \big[\sup_{t\in [0,T]} \Vert u(t) \Vert_{L^2_{\lambda}}^p \big] \leq  C_{p, T,\lambda} <\infty.
\end{align}

To prove \eqref{260210.1127}, we can and will assume  $g\in C_c^{\infty}(\mathbb{R})$, because the space $C_c^{\infty}(\mathbb{R})$ is dense in $L^2_{\lambda}(\mathbb{R})$ for any $\lambda>0$.
In the sequel, we fix an arbitrary function $g\in C_c^{\infty}(\mathbb{R})$ and  prove \eqref{260210.1127}. The proof is divided into three steps.
\vskip 0.2cm

\textbf{Step 1.} For $\varepsilon>0$,  we construct a random controlled equation for a random field $Y^\varepsilon$ with a specially designed control $h^\varepsilon$ so that $Y^\varepsilon(T)$ can be made arbitrarily close to the function $g$.

Set
\begin{align*}
  \rho(t,x) = e^{-t|x|^2}, \quad t>0, x\in\mathbb{R}.
\end{align*}
Let $t_1\in (0,T)$ be a parameter to be determined, which will be sufficiently close to $T$. Let $\varepsilon>0$ be another parameter to be determined, which will be sufficiently close to $0$. Set
\begin{align}\label{260312.1051}
  \hat{u}^{\varepsilon}(t_1,x,\omega) :=  \mathbf{1}_{\{\Vert u(t_1) \Vert_{L^2_{\lambda}}\leq\frac{1}{\varepsilon}\}}(\omega)\big(u(t_1,\cdot,\omega)\rho(T-t_1,\cdot)\big)*\phi_{\varepsilon}(x),
\end{align}
where $\phi$ is a nonnegative smooth function supported on $[-1,1]$ and satisfying $\int_{\mathbb{R}}\phi(x) dx =1$, and where $\phi_{\varepsilon}(x) := \varepsilon^{-1}\phi(x/\varepsilon)$ is the usual mollifier.
Then $\hat{u}^{\varepsilon}(t_1,\cdot,\omega)$ is smooth on $\mathbb{R}$ for $\mathbb{P}$-a.e. $\omega$.

Consider the equation:
\begin{align}\label{260318.1659}
\begin{cases}
    \partial_t Z_1(t,x) = \frac{1}{2}\partial_{xx} Z_1(t,x) + b(Z_1(t,x)) , \quad t\in [t_1, T],\\
    Z_1(t_1,x) = 0.
\end{cases}
\end{align}
Equation (\ref{260318.1659}) has a unique solution $Z_1\in C([t_1, T], L^2_{\lambda})$ which admits the representation
\[
Z_1(t)=\int_{t_1}^tP_{t-s} b(Z_1(s))ds.
\]
Hence, by condition (H1) and (\ref{260208.1019}), we have
\begin{equation}\label{260330.1432}
  \begin{aligned}
  \Vert Z_1(T)\Vert_{L^2_\lambda} \leq &  \int_{t_1}^T \Vert P_{T-s} b(Z_1(s)) \Vert_{L^2_{\lambda}} ds \leq e^{\frac{\lambda^2 T}{2}} \int_{t_1}^T [\Vert b(0) \Vert_{L^2_{\lambda}} + L_b \Vert Z_1(s) \Vert_{L^2_{\lambda}}] ds \\
  \leq & C_{T,\lambda} \times (T-t_1).
\end{aligned}
\end{equation}
$Z_1(T)$ will go to zero as $t_1\rightarrow T$. 
With $Z_1$ in hand, for $\varepsilon>0$, $t\in [t_1, T]$ and $x\in\mathbb{R}$, set
\begin{align}\label{260210.2201}
  Z_2^{\varepsilon}(t,x) := \frac{T-t}{T-t_1} \hat{u}^{\varepsilon}(t_1,x) + \frac{t-t_1}{T-t_1} g(x),
\end{align}
and
\begin{equation}\label{260330.1423}
\begin{aligned}
  h^{\varepsilon}(t,x) := &\partial_t Z_2^{\varepsilon}(t,x) - \frac{1}{2}\partial_{xx} Z_2^{\varepsilon}(t,x) - [b(Z_1(t,x)+ Z_2^{\varepsilon}(t,x)) - b(Z_1(t,x))] .
\end{aligned}
\end{equation}
Then $Z_2^{\varepsilon}$ satisfies the terminal condition $Z_2^{\varepsilon}(T,x) = g(x)$ and solves the following random PDE:
\begin{align*}
\begin{cases}
    \partial_t Z_2^{\varepsilon}(t,x) = \frac{1}{2}\partial_{xx} Z_2^{\varepsilon}(t,x) + b(Z_1(t,x)+ Z_2^{\varepsilon}(t,x)) - b(Z_1(t,x)) + h^{\varepsilon}(t,x), \quad t\in [t_1, T],\\
    Z_2^{\varepsilon}(t_1,x) = \hat{u}^{\varepsilon}(t_1,x).
\end{cases}
\end{align*}
%
Now we define the random field
$Y^\varepsilon$ as:
\begin{align}\label{260330.1442}
  Y^{\varepsilon}(t,x) := Z_1(t,x) + Z_2^{\varepsilon}(t,x), \quad t\in [t_1, T].
\end{align}
Then $Y^{\varepsilon}$ and $h^{\varepsilon}$ satisfy the following random controlled equation:
\begin{align}\label{260211.1659}
\begin{cases}
    \partial_t Y^{\varepsilon}(t,x) = \frac{1}{2}\partial_{xx} Y^{\varepsilon}(t,x) + b(Y^{\varepsilon}(t,x)) + h^{\varepsilon}(t,x), \quad t\in [t_1, T],\\
    Y^{\varepsilon}(t_1,x) = \hat{u}^{\varepsilon}(t_1,x).
\end{cases}
\end{align}
Moreover, $Y^{\varepsilon}(T)$ will be arbitrarily close to the function $g$ as $t_1$ is getting close to $T$.

For the control $ h^{\varepsilon}(t,x)$, we have the following estimate which will be used later.
\begin{align}\label{260211.1611}
  \Vert h^{\varepsilon}(t) \Vert_{L^2} \leq C_{\varepsilon, T - t_1,\lambda}, \quad \forall\  t\in [t_1,T], \quad \mathbb{P}\text{-a.s.},
\end{align}
where $L^2:=L^2(\mathbb{R})$, $C_{\varepsilon, T - t_1,\lambda}$ is a constant  independent of $\omega$.

Indeed, by (\ref{260330.1423}), (\ref{260210.2201}) and the Lipschitz property of $b$, we have
\begin{equation*}
  \begin{aligned}
      | h^{\varepsilon}(t,x)|        \leq  & \frac{ |g(x)- \hat{u}^{\varepsilon}(t_1,x)|}{T-t_1} +  \frac{1}{2}\left(\frac{T-t}{T-t_1}|\Delta\hat{u}^{\varepsilon}(t_1,x)| + \frac{t-t_1}{T-t_1}|\Delta g(x)|\right) + L_b |Z_2^{\varepsilon}(t,x)|.
  \end{aligned}
\end{equation*}
In view of the definition of $Z_2^{\varepsilon}$, to prove (\ref{260211.1611}), it suffices to show that there exists a constant $C_{\varepsilon}$, independent of $\omega$, such that
\begin{align*}
  \Vert \hat{u}^{\varepsilon}(t_1)\Vert_{L^2}^2 \leq C_{\varepsilon} \quad \text{ and }\quad \Vert \Delta\hat{u}^{\varepsilon}(t_1)\Vert_{L^2}^2 \leq C_{\varepsilon}.
\end{align*}
In fact, by Young's inequality,
\begin{equation*}
\begin{aligned}
  \Vert \big(u(t_1)\rho(T-t_1)\big)*\phi_{\varepsilon} \Vert_{L^2}^2
  \leq & \Vert u(t_1) \rho(T-t_1)\Vert_{L^2}^2 \Vert \phi_{\varepsilon} \Vert_{L^1}^2 \\
\leq & \sup_{x\in\mathbb{R}} \Big[ e^{-2(T-t_1)|x|^2+\lambda|x|} \Big]\times\int_{\mathbb{R}} |u(t_1,x)|^2 e^{-\lambda|x|} dx \\
\leq & C_{T-t_1, \lambda}\Vert u(t_1)\Vert_{L^2_{\lambda}}^2.
\end{aligned}
\end{equation*}
Hence, by the definition of $\hat{u}^{\varepsilon}(t_1)$,
\begin{align*}
   \Vert \hat{u}^{\varepsilon}(t_1)\Vert_{L^2}^2 \leq \frac{1}{\varepsilon^2} C_{T-t_1, \lambda} ,\quad \mathbb{P}\text{-a.s.}
\end{align*}
Since $\phi_{\varepsilon}$ is a mollifier, a similar argument shows that there exists a constant $C_{\varepsilon}>0$ such that, for $\mathbb{P}$-a.s.,
\begin{align*}
  \Vert\Delta\hat{u}^{\varepsilon}(t_1)\Vert_{L^2}^2 \leq C_{\varepsilon} C_{T-t_1, \lambda}.
\end{align*}
This proves (\ref{260211.1611}).
\vskip 0.2cm

\textbf{Step 2.} 
For $\varepsilon>0$, we construct the solution $u^\varepsilon$ of a stochastic controlled equation, whose law will be absolutely continuous with respect to that of the solution $u$.

For $t\in [0,t_1]$, set
\begin{align*}
  u^{\varepsilon}(t,x) := u(t,x), \quad  x\in\mathbb{R},
\end{align*}
and for $t\in [t_1,T]$, let $u^{\varepsilon}(t,x)$ be the solution of
\begin{equation}\label{260211.1700}
\begin{aligned}
    u^{\varepsilon}(t,x) := &  P_{t-t_1} u(t_1)(x) + \int_{t_1}^t\int_{\mathbb{R}} p_{t-s}(x,y)b(u^{\varepsilon}(s,y)) dsdy \\
& + \int_{t_1}^t\int_{\mathbb{R}} p_{t-s}(x,y)\sigma(u^{\varepsilon}(s,y)) W(ds,dy)  \\
& + \int_{t_1}^t\int_{\mathbb{R}} p_{t-s}(x,y)h^{\varepsilon}(s,y) dsdy, \quad  x\in\mathbb{R}.
\end{aligned}
\end{equation}

Since $u$ satisfies (\ref{3.1}), it follows that, for any $t\geq 0$,
\begin{equation}\label{260211.1650}
\begin{aligned}
    u^{\varepsilon}(t,x) := &  P_t f(x) + \int_{0}^t\int_{\mathbb{R}} p_{t-s}(x,y)b(u^{\varepsilon}(s,y)) dsdy \\
& + \int_{0}^t\int_{\mathbb{R}} p_{t-s}(x,y)\sigma(u^{\varepsilon}(s,y)) W(ds,dy) \\
& + \int_{0}^t\int_{\mathbb{R}} p_{t-s}(x,y)h^{\varepsilon}(s,y)\mathbf{1}_{[t_1,T]}(s) dsdy ,\quad x\in\mathbb{R}.
\end{aligned}
\end{equation}
For $t\geq 0$ and $x\in\mathbb{R}$, let
\[
H^{\varepsilon}(t,x) := \sigma(u^{\varepsilon}(t,x))^{-1} h^{\varepsilon}(t,x)\mathbf{1}_{[t_1,T]}(t),
\]
and
\[
\hat{W}^{\varepsilon}(t,x) := W(t,x) + \int_0^t\int_0^x H^{\varepsilon}(s,y) dsdy .
\]
It follows from (\ref{260211.1611}) and condition (H3) that
\begin{align*}
  \int_0^T\int_{\mathbb{R}} |H^{\varepsilon}(s,y)|^2 dsdy \leq C_{\varepsilon,T-t_1,\lambda}, \quad \mathbb{P}\text{-a.s.}
\end{align*}
Hence the Novikov condition holds for each fixed $\varepsilon$. By the Girsanov theorem, $\hat{W}^{\varepsilon}(t,x)$ is a Brownian sheet under probability measure $\hat{\mathbb{P}}_T^{\varepsilon}$ defined by
\begin{align*}
  d\hat{\mathbb{P}}_T^{\varepsilon} := \exp\left(-\int_{0}^{T}\int_{\mathbb{R}} H^{\varepsilon}(s,y) W(ds,dy) -\frac{1}{2} \int_{0}^{T}\int_{\mathbb{R}} |H^{\varepsilon}(s,y)|^2 dsdy \right) d\mathbb{P}.
\end{align*}
Therefore, (\ref{260211.1650}) implies that for any $t\in [0,T]$, the law of $u^{\varepsilon}(t)$ under $\hat{\mathbb{P}}_T^{\varepsilon}$ coincides with the law of $u(t)$ under the original measure $\mathbb{P}$.
 Because the two probability measures $\hat{\mathbb{P}}_T^{\varepsilon}$ and $\mathbb{P}$ are equivalent, the law of $u^{\varepsilon}$ is equivalent to the law of $u$ on $C([0, T], L^2_{\lambda})$.
\vskip 0.2cm

\textbf{Step 3.} Completion of the proof of \eqref{260210.1127}:
$$\mathbb{P}(\Vert u(T) - g\Vert_{L^2_{\lambda}} <r)>0.$$
As the law of $u^{\varepsilon}$ is equivalent to the law of $u$,  to prove \eqref{260210.1127}, it is enough to prove
\begin{align}\label{260211.1658}
   \mathbb{P}(\Vert u^{\varepsilon}(T) - g\Vert_{L^2_{\lambda}} \geq r)<1.
\end{align}
As mention in Step 1,  $Y^{\varepsilon}(T)$ can be
made arbitrarily close to the function $g$ by choosing $t_1$ sufficiently close to $T$. Hence, to prove (\ref{260211.1658}), we need to estimate  $u^\varepsilon-Y^\varepsilon$. By (\ref{260211.1659}) and (\ref{260211.1700}), for $t\geq t_1$,
\begin{equation}\label{260310.1112}
  \begin{aligned}
  \mathbb{E} \Vert u^{\varepsilon}(t) - Y^{\varepsilon}(t)\Vert_{L^2_{\lambda}}^2 \leq & 3 \mathbb{E} \Vert P_{t-t_1}\big(\hat{u}^{\varepsilon}(t_1) - u(t_1) \big)\Vert_{L^2_{\lambda}}^2 \\
  & + 3 \mathbb{E} \left\Vert \int_{t_1}^t\int_{\mathbb{R}} p_{t-s}(x,y) [ b(u^{\varepsilon}(s,y)) - b(Y^{\varepsilon}(s,y)) ] dsdy \right\Vert_{L^2_{\lambda}}^2 \\
    & + 3 \mathbb{E} \left\Vert \int_{t_1}^t\int_{\mathbb{R}} p_{t-s}(x,y) \sigma(u^{\varepsilon}(s,y)) W(ds,dy) \right\Vert_{L^2_{\lambda}}^2  \\
    =:& I+ II + III.
\end{aligned}
\end{equation}
By (\ref{260208.1019}),
\begin{equation}
  \begin{aligned}
    I\leq 3  e^{\frac{\lambda^2 (t-t_1)}{2}} \mathbb{E} \Vert \hat{u}^{\varepsilon}(t_1) - u(t_1)\Vert_{L^2_{\lambda}}^2.
  \end{aligned}
\end{equation}
By (H1), H\"{o}lder's inequality, Fubini's theorem, and (\ref{260207.1152}),
\begin{equation}\label{260310.1113}
  \begin{aligned}
  II = & 3\mathbb{E} \int_{\mathbb{R}}\left|\int_{t_1}^t\int_{\mathbb{R}} p_{t-s}(x,y) [ b(u^{\varepsilon}(s,y)) - b(Y^{\varepsilon}(s,y)) ] dsdy\right|^2 e^{-\lambda|x|} dx \\
  \leq & 3(t-t_1)\mathbb{E} \int_{\mathbb{R}} \int_{t_1}^t\int_{\mathbb{R}} p_{t-s}(x,y) L_b^2| u^{\varepsilon}(s,y) - Y^{\varepsilon}(s,y) |^2  e^{-\lambda|x|} dsdy dx \\
  \leq & 3(t - t_1)  \int_{t_1}^t\int_{\mathbb{R}}  L_b^2 e^{\frac{\lambda^2}{2}(t-s)} \mathbb{E}| u^{\varepsilon}(s,y) - Y^{\varepsilon}(s,y) |^2  e^{-\lambda|y|} dsdy \\
  \leq & 3  L_b^2 (T- t_1) e^{\frac{\lambda^2}{2}(T-t_1)} \int_{t_1}^t \mathbb{E} \Vert u^{\varepsilon}(s) - Y^{\varepsilon}(s) \Vert_{L^2_{\lambda}}^2 ds .
  \end{aligned}
\end{equation}
For the term $III$, we have
\begin{equation}\label{260310.1114}
  \begin{aligned}
    III = & 3 \mathbb{E} \bigg\Vert \int_{t_1}^t\int_{\mathbb{R}} p_{t-s}(x,y) [ \sigma(u^{\varepsilon}(s,y)) -\sigma(Y^{\varepsilon}(s,y)) ] W(ds,dy) \\
    & + \int_{t_1}^t\int_{\mathbb{R}} p_{t-s}(x,y) \sigma(Y^{\varepsilon}(s,y))  W(ds,dy) \bigg\Vert_{L^2_{\lambda}}^2 \\
    \leq & 6 \mathbb{E} \left\Vert \int_{t_1}^t\int_{\mathbb{R}} p_{t-s}(x,y) [ \sigma(u^{\varepsilon}(s,y)) -\sigma(Y^{\varepsilon}(s,y)) ] W(ds,dy) \right\Vert_{L^2_{\lambda}}^2 \\
    & + 6 \mathbb{E} \left\Vert \int_{t_1}^t\int_{\mathbb{R}} p_{t-s}(x,y)\sigma(Y^{\varepsilon}(s,y))  W(ds,dy) \right\Vert_{L^2_{\lambda}}^2 \\
    =: & III_1 + III_2 .
  \end{aligned}
\end{equation}
By the It\^{o} isometry, (\ref{260207.2020}) and (H2), we have
\begin{equation}\label{260310.1115}
  \begin{aligned}
    III_1 = & 6\mathbb{E} \int_{t_1}^t\int_{\mathbb{R}} \Vert p_{t-s}(\cdot,y)\Vert^2_{L^2_{\lambda}} | \sigma(u^{\varepsilon}(s,y)) -\sigma(Y^{\varepsilon}(s,y)) |^2 dsdy \\
    & \leq C \mathbb{E} \int_{t_1}^t \frac{1}{\sqrt{t-s}} e^{\frac{\lambda^2 (t-s)}{4}} \int_{\mathbb{R}} | \sigma(u^{\varepsilon}(s,y)) -\sigma(Y^{\varepsilon}(s,y)) |^2 e^{-\lambda|y|} dyds \\
    & \leq C L_{\sigma}^2 e^{\frac{\lambda^2 T}{4}}  \int_{t_1}^t \frac{1}{\sqrt{t-s}}  \mathbb{E} \Vert u^{\varepsilon}(s) - Y^{\varepsilon}(s) \Vert_{L^2_{\lambda}}^2 ds .
  \end{aligned}
\end{equation}
Similarly,
\begin{equation}\label{260309.1610}
  \begin{aligned}
    III_2 = & 6\mathbb{E} \int_{t_1}^t\int_{\mathbb{R}} \Vert p_{t-s}(\cdot,y)\Vert^2_{L^2_{\lambda}} | \sigma(Y^{\varepsilon}(s,y)) |^2 dsdy \\
    & \leq C  e^{\frac{\lambda^2 T}{4}} \int_{t_1}^t \frac{1}{\sqrt{t-s}}  \left( \frac{|\sigma(0)|^2}{\lambda} + L_{\sigma}^2 \mathbb{E} \Vert  Y^{\varepsilon}(s) \Vert_{L^2_{\lambda}}^2 \right) ds .
  \end{aligned}
\end{equation}
A direct calculation gives
\begin{align}\label{260310.1110}
      \Vert f\ast \phi_{\varepsilon}\Vert_{L^2_{\lambda}} \leq e^{\frac{\lambda \varepsilon}{2}} \Vert f \Vert_{L^2_{\lambda}}.
\end{align}
By (\ref{260312.1051}), (\ref{260310.1110}) and (\ref{260313.1659}), there exists a constant $C$ such that
\begin{align}\label{260604.2309}
  \sup_{\varepsilon\in (0,1]}\mathbb{E}\Vert \hat{u}^{\varepsilon}(t_1)\Vert^2_{L^2_{\lambda}} \leq C \mathbb{E}\Vert u(t_1)\Vert^2_{L^2_{\lambda}} <\infty.
\end{align}
By (\ref{260330.1442}), (\ref{260330.1432}), (\ref{260210.2201}) and (\ref{260604.2309}),
we see that
\begin{align*}
  \sup_{\varepsilon\in (0,1]}\sup_{s\in [t_1, T]} \mathbb{E} \Vert Y^{\varepsilon} (s) \Vert^2_{L^2_{\lambda}} \leq C_T <\infty.
\end{align*}
It follows from (\ref{260309.1610}) that there exists a constant $C$, independent of $\varepsilon$, such that for all $\varepsilon\in (0,1)$,
\begin{align}\label{260310.1116}
  III_2 \leq C (T-t_1)^{\frac{1}{2}}.
\end{align}
Combining (\ref{260310.1112})--(\ref{260310.1115}) and (\ref{260310.1116}) gives
\begin{equation*}
  \begin{aligned}
  \mathbb{E} \Vert u^{\varepsilon}(t) - Y^{\varepsilon}(t)\Vert_{L^2_{\lambda}}^2 \leq & Ce^{\frac{\lambda^2 (t-t_1)}{2}}\mathbb{E} \Vert \hat{u}^{\varepsilon}(t_1) - u(t_1) \Vert_{L^2_{\lambda}}^2 + C (T-t_1)^{\frac{1}{2}} \\
  & + C (T- t_1)  \int_{t_1}^t \mathbb{E} \Vert u^{\varepsilon}(s) - Y^{\varepsilon}(s) \Vert_{L^2_{\lambda}}^2 ds \\
  & + C \int_{t_1}^t \frac{1}{\sqrt{t-s}}  \mathbb{E} \Vert u^{\varepsilon}(s) - Y^{\varepsilon}(s) \Vert_{L^2_{\lambda}}^2 ds .
  \end{aligned}
\end{equation*}
By Gronwall's inequality (see, e.g., Lemma 3.2 of \cite{SWZ25}), for any $t\in [t_1, T]$,
\begin{align}\label{260330.1444}
  \mathbb{E} \Vert u^{\varepsilon}(t) - Y^{\varepsilon}(t)\Vert_{L^2_{\lambda}}^2 \leq C_T \left[ \mathbb{E} \Vert \hat{u}^{\varepsilon}(t_1) - u(t_1) \Vert_{L^2_{\lambda}}^2 + (T-t_1)^{\frac{1}{2}} \right] e^{C_T} .
\end{align}

Next, we prove that for any $\eta>0$, one can first choose $t_1\in (0,T)$ sufficiently close to $T$, and then choose $\varepsilon>0$ sufficiently small, such that
\begin{align}\label{260314.0825}
  \mathbb{E} \Vert \hat{u}^{\varepsilon}(t_1) - u(t_1) \Vert_{L^2_{\lambda}}^2 <\eta.
\end{align}
If (\ref{260314.0825}) is proved, then by (\ref{260330.1442}), (\ref{260330.1444}) and (\ref{260330.1432}), we can choose $\eta$ and then $t_1, \varepsilon$ so that
\begin{equation*}
  \begin{aligned}
  \mathbb{P}(\Vert u^{\varepsilon}(T) - g\Vert_{L^2_{\lambda}}\geq r)
  \leq & \frac{1}{r^2} \mathbb{E} [\Vert u^{\varepsilon}(T) - g\Vert_{L^2_{\lambda}}^2] \\
  = & \frac{1}{r^2} \mathbb{E} [\Vert u^{\varepsilon}(T) - Y^{\varepsilon}(T) + Y^{\varepsilon}(T) - Z_2^{\varepsilon}(T)\Vert_{L^2_{\lambda}}^2] \\
  = & \frac{2}{r^2} \mathbb{E} [\Vert u^{\varepsilon}(T) - Y^{\varepsilon}(T)\Vert_{L^2_{\lambda}}^2 + \Vert Z_1(T)\Vert_{L^2_{\lambda}}^2] \\
  \leq & \frac{C_{\lambda, T}}{r^2} \left[ \mathbb{E} \Vert \hat{u}^{\varepsilon}(t_1) - u(t_1) \Vert_{L^2_{\lambda}}^2 + (T-t_1)^{\frac{1}{2}} \right] \\
  < & 1.
\end{aligned}
\end{equation*}
Thus, (\ref{260211.1658}) is proved once (\ref{260314.0825}) is established.

Now it remains to prove (\ref{260314.0825}). By (\ref{260312.1051}), we have
\begin{equation}\label{260313.1737}
  \begin{aligned}
    & \mathbb{E} \Vert \hat{u}^{\varepsilon}(t_1) - u(t_1) \Vert_{L^2_{\lambda}}^2 \\
    \leq & 3 \mathbb{E} \Vert \hat{u}^{\varepsilon}(t_1) - \mathbf{1}_{\{\Vert u(t_1) \Vert_{L^2_{\lambda}}\leq\frac{1}{\varepsilon}\}} u(t_1)\rho(T-t_1) \Vert_{L^2_{\lambda}}^2 \\
    & + 3 \mathbb{E} \Vert  \mathbf{1}_{\{\Vert u(t_1) \Vert_{L^2_{\lambda}}\leq\frac{1}{\varepsilon}\}} u(t_1)[ \rho(T-t_1) -1 ] \Vert_{L^2_{\lambda}}^2 \\
    & + 3 \mathbb{E} \Vert  [\mathbf{1}_{\{\Vert u(t_1) \Vert_{L^2_{\lambda}}\leq\frac{1}{\varepsilon}\}} -1 ] u(t_1) \Vert_{L^2_{\lambda}}^2 \\
    =: & J_1(t_1,\varepsilon) + J_2(t_1,\varepsilon) + J_3(t_1,\varepsilon).
  \end{aligned}
\end{equation}
For the term $J_2$, note that
\begin{equation*}
  \begin{aligned}
  \Vert  u(t_1)[ \rho(T-t_1) -1 ] \Vert_{L^2_{\lambda}} \leq & \Vert  [u(t_1) - u(T)][ \rho(T-t_1) -1 ] \Vert_{L^2_{\lambda}} + \Vert  u(T)[ \rho(T-t_1) -1 ] \Vert_{L^2_{\lambda}} \\
  \leq & \Vert  u(t_1) - u(T) \Vert_{L^2_{\lambda}} + \Vert  u(T)[ \rho(T-t_1) -1 ] \Vert_{L^2_{\lambda}} .
\end{aligned}
\end{equation*}
Since $u\in C([0,\infty), L^2_{\lambda})$, $\mathbb{P}$-a.s., the dominated convergence theorem gives, $\mathbb{P}$-a.s.,
\begin{align*}
  \lim_{t_1\rightarrow T} \Vert  u(t_1)[ \rho(T-t_1) -1 ] \Vert_{L^2_{\lambda}} = 0.
\end{align*}
By (\ref{260313.1659}) and the fact that 
\[
\Vert  u(t_1)[ \rho(T-t_1) -1 ] \Vert_{L^2_{\lambda}}^2 \leq \Vert  u(t_1) \Vert_{L^2_{\lambda}}^2,
\]
we see that $\{\Vert  u(t_1)[ \rho(T-t_1) -1 ] \Vert_{L^2_{\lambda}}\}_{t_1\in [0,T]}$ is uniformly integrable in $L^2(\Omega)$. Hence, by Vitali's convergence theorem, we obtain
\begin{align*}
  \lim_{t_1\rightarrow T}\sup_{\varepsilon>0} J_2(t_1,\varepsilon) \leq 3 \lim_{t_1\rightarrow T} \mathbb{E} \Vert  u(t_1)[ \rho(T-t_1) -1 ] \Vert_{L^2_{\lambda}}^2 = 0.
\end{align*}
Hence, we can choose $t^*\in (0,T)$ sufficiently close to $T$, such that
\begin{align}\label{260314.0822}
  J_2(t^*,\varepsilon)\leq \frac{\eta}{3},\quad \forall\  \varepsilon>0.
\end{align}
For this choice of $t^*$, the dominated convergence theorem and (\ref{260313.1659}) yield
\begin{align*}
  \lim_{\varepsilon\rightarrow 0}J_3(t^*,\varepsilon) = 3 \lim_{\varepsilon\rightarrow 0}\mathbb{E} \Vert  [\mathbf{1}_{\{\Vert u(t^*) \Vert_{L^2_{\lambda}}\leq\frac{1}{\varepsilon}\}} -1 ] u(t^*) \Vert_{L^2_{\lambda}}^2 =0.
\end{align*}
Hence, we can choose $\varepsilon_3>0$ sufficiently small such that
\begin{align}\label{260314.0823}
  J_3(t^*,\varepsilon) \leq \frac{\eta}{3}, \quad \forall\ 0<\varepsilon<\varepsilon_3. 
\end{align}
By (\ref{260310.1110}) and the fact that
\begin{align*}
  \Vert h\ast \phi_{\varepsilon} - h\Vert_{L^2_{\lambda}} \rightarrow 0, \quad \forall\ h\in L^2_{\lambda},
\end{align*}
we deduce that
\begin{align*}
  \lim_{\varepsilon\rightarrow 0 } J_1(t^*,\varepsilon) \leq  3\lim_{\varepsilon\rightarrow 0} \mathbb{E} \Vert \big(u(t^*)\rho(T-t^*)\big)*\phi_{\varepsilon} - u(t^*)\rho(T-t^*) \Vert_{L^2_{\lambda}}^2 = 0.
\end{align*}
Therefore, we can choose $0<\varepsilon^*<\varepsilon_3$ sufficiently small such that
\begin{align}\label{260314.0824}
   J_1(t^*,\varepsilon^*) \leq \frac{\eta}{3}.
\end{align}
Combining (\ref{260313.1737})--(\ref{260314.0824}) yields (\ref{260314.0825}) with $t_1=t^*$ and $\varepsilon=\varepsilon^*$. This completes the proof of \eqref{260211.1658}, and hence the proof of Theorem \ref{260312.1046}. 
\end{proof}

%
%
%
%
%
%
%
%
%
%
%
%
%

\vskip 0.3cm
\noindent{\large\bf Acknowledgements}
This work is partially supported by the National Key R\&D Program of China (No. 2022YFA1006001), the National Natural Science Foundation of China (Nos. 12131019, 12571158, 12371151), and the Fundamental Research Funds for the Central Universities (No. WK0010000081).

\bibliographystyle{plain}

\end{document}